\documentclass[a4paper,12pt]{amsart}
\usepackage{amsmath}
\usepackage{amsthm}
\usepackage{amssymb}
\usepackage{fullpage} 
\usepackage{bbm} 
\usepackage{mathtools} 

\mathtoolsset{showonlyrefs}

\newtheorem{cor}{Corollary}
\newtheorem{thm}{Theorem}
\newtheorem{rem}{Remark}

\newtheorem{prop}{Proposition}
\newtheorem{lem}{Lemma}

\newcommand{\sumstar}{\;\sideset{}{^*}\sum}
\newcommand{\sumb}{\sideset{}{^\flat}\sum}

\newcommand{\sump}{\sideset{}{^P}\sum}

\newcommand{\ga}{\gamma}
\newcommand{\de}{\delta}

\newcommand{\tRe}{\textup{Re }}
\newcommand{\tReb}{\textup{Re}}
\newcommand{\tIm}{\textup{Im }}
\newcommand{\bfrac}[2]{\left(\frac{#1}{#2}\right)}
\newcommand{\M}{\mathcal M}

\newcommand{\E}{\mathcal E}

\newcommand{\B}{\mathcal B}

\newcommand{\C}{\mathcal C}
\newcommand{\N}{\mathcal N}
\newcommand{\A}{{\mathcal A}}
\newcommand{\R}{\mathcal R}
\newcommand{\U}{\mathcal U}
\newcommand{\modd}[1]{\; ( \textup{mod} \; #1)}
\newcommand{\q}{\mathfrak q}

\newcommand{\tZ}{\tilde{Z}}

\newcommand{\Cf}{\mathfrak C}

\newcommand{\Df}{\mathfrak D}

\begin{document}
\title{Prime values of $a^2 + p^4$}
\author{D.R. Heath-Brown}
\address{Mathematical Institute \\
University of Oxford\\
Andrew Wiles Building \\
Radcliffe Observatory Quarter\\
Woodstock Road\\
Oxford\\UK\\
OX2 6GG }
\email{Roger.Heath-Brown@maths.ox.ac.uk}
\author{Xiannan Li}
\address{Mathematical Institute \\
University of Oxford\\
Andrew Wiles Building \\
Radcliffe Observatory Quarter\\
Woodstock Road\\
Oxford\\UK\\
OX2 6GG }
\email{lix1@maths.ox.ac.uk}


\begin{abstract} 
We prove an asymptotic formula for the number of primes of the shape
$a^2+p^4$, thereby refining the well known work of Friedlander and
Iwaniec \cite{FI}.  Along the way, we prove a result on equidistribution of primes up to $x$, in which the moduli may be
almost as large as $x^2$.  
\end{abstract}

\maketitle
\section{Introduction}
Many remarkably difficult conjectures in prime number theory take the
form that there are infinitely many primes in some set of natural
numbers $S$.  In many interesting examples, we even have conjectured
asymptotic formulas for the number of primes in $S$.  Thus, we think
that there are infinitely many primes of the form $p+2$, $a^2+1$, $a^2
+ b^6$, and so on.  Here, and everywhere in the paper, $p$ shall
always denote a prime. 

The generality of our belief is in stark contrast with the paucity of
examples for which we can prove our conjectures.  In this paper, we
are interested in the problem of finding primes in sequences which
occur as the special values of a polynomial in two variables.  For
polynomials in one variable, only the linear case is understood, from
the work of Dirichlet. 

A classical result is that there are infinitely many primes of the
form $a^2+b^2$.  Indeed, by a result of Fermat, primes of that form
are essentially the same as primes of the form $4n+1$, so that this
reduces to a special case of Dirichlet's theorem on primes in
arithmetic progressions.  Let us define the exponential density of the
sequence of values of the polynomial $P(a, b)$ as the infimum of those
real $\lambda$ for which
\begin{equation}
\#\{P(a, b) \leq x\} \ll x^\lambda.
\end{equation}

Then the density of the sequence defined by $a^2+b^2$ is $1$, the same
as the set of all natural numbers.   

It is much more challenging to prove a similar result when the
sequence given by $P(a, b)$ has density less than $1$.  The first
result in this direction was the breakthrough of Friedlander and
Iwaniec \cite{FI} on the prime values of $a^2+b^4$, which was followed
by the result of Heath-Brown \cite{HB} on primes values of $a^3+2b^3$.
It is worth mentioning that the density of the sequence is not the
only measure of difficulty, as no results are available for prime
values of $a^2+b^3$ due to its lack of structure.   

Aside from generalizations of Heath-Brown's result to more general
cubic polynomials by Heath-Brown and Moroz \cite{HBM}, the theorems of
Friedlander and Iwaniec \cite{FI} and Heath-Brown \cite{HB} remain the
only results of this type.  In this paper, we add the following
example on prime values of $a^2+p^4$. 

\begin{thm}\label{thm:1}
\begin{equation}
\#\{a^2+p^4 \leq x: a>0 \textup{ and } a^2 + p^4 \textup{ is prime}\}
= \nu \frac{4Jx^{3/4}}{\log^2 x}\left(1+O_\epsilon\bfrac{1}{(\log
    x)^{1-\epsilon}}\right), 
\end{equation}where
\begin{equation}
J = \int_0^1 \sqrt{1-t^4} dt,
\end{equation}and
\begin{equation}
\nu = \prod_p \left(1-\frac{\chi_4(p)}{p-1}\right),
\end{equation}for $\chi_4$ the non-principal character modulo $4$.
\end{thm}
In the statement of the Theorem above, and in the rest of the paper,
$\epsilon$ denotes any sufficiently small, positive constant, not
necessarily the same in each occurrence. 

\begin{rem}
The reader may check that the number of elements $a^2+p^4 \leq x$ is
well approximated by $\frac{4 J x^{3/4}}{\log x}$. 

Our method also
applies to more general sequences $a^2+y^4$ where $y$ is restricted to
a set $Y$ which is regularly distributed and is not too sparse. 
\end{rem}

As in the works \cite{FI} and \cite{HB}, the proof rests on
establishing a level of distribution for the sequence, and the
estimation of special bilinear sums.  The first appears in classical
sieve theory, and simply asks for good estimates for the remainder
term in counting numbers of the form $a^2+p^4$ divisible by a given
integer $d$, averaged over $d$.  The bilinear sums estimate is the
ingredient which allows us to overcome the parity barrier, and as in
previous works is the most significant part of the proof.   

One ingredient used in Friedlander and Iwaniec's work on prime values
of $a^2+b^4$ is the regularity of the distribution of the squares of
integers, which allows them to use a delicate harmonic analysis
argument to extract certain main terms in their bilinear sum, and
prove that the error terms are small on average (see Sections 4-9 in
\cite{FI}).  This is the portion of their work which overcomes the
sparsity of their sequence. 

This regularity does not exist in the case of squares of primes, and
we need to develop a method which  applies for more general sequences.
In particular, we prove a result about the distribution of sequences
in arithmetic progressions, which is similar in spirit to a general
form of the Barban--Davenport--Halberstam theorem.  In our case, for
$a^2 + p^4 \le x$, we have $p \le x^{1/4}$, while the modulus
appearing in our bilinear sum can be as large as $x^{1/2 - \delta}$.
We thus need an equidistribution result which holds when the modulus
goes up to nearly the square of the length of the sum, in contrast to
the Barban--Davenport--Halberstam theorem.  Since this result is of
independent interest, we first illustrate our result in the case of
primes below. 

\begin{cor}\label{cor:primes}
For $(a, q) = 1$, let 
\begin{equation}
S(x; a, q) = \sum_{\substack{m, n \leq x \\n \equiv am \modd{q}}} \Lambda(m) \Lambda(n).
\end{equation}
Then for any $A>0$, there exists $B = B(A)$ such that
\begin{equation}
\sum_{q\leq Q}\; \sumstar_{a \modd{q}} \left|S(x; a, q) -
  \frac{x^2}{\phi(q)}\right|^2 \ll \frac{x^4}{\log^A x} 
\end{equation}for $Q \leq x^2 (\log x)^{-B}$.
\end{cor}

At first sight it seems remarkable that the primes up to $x$ should be
well distributed for moduli as large as $x^2 (\log x)^{-B}$, but this
is essentially the conclusion of the Corollary.  The reader may verify
however that the analogous result does not hold without the average
over $a \modd{q}$. 

For our application, we will require a more general result which
applies for sequences satisfying the Siegel--Walfisz condition, which
is essentially saying that the sequence is equidistributed for small
moduli. Remark \ref{rem:SWdef} below makes this precise.   

\begin{rem}\label{rem:SWdef}
We say that an arithmetic function $c(n)$ satisfies a Siegel--Walfisz
condition if for any constant $\kappa$, we have that for any $q <
(\log x)^{\kappa}$, and any nonprincipal character $\chi \modd{q}$, we have
\begin{equation}\label{eqn:SW1}
\sum_{n\leq x} \chi(n)c(n) \ll_\kappa x^{1/2}
\|c \| (\log x)^{-\kappa}. 
\end{equation}
\end{rem}
Here $\|c \|^2 = \sum_{n} |c(n)|^2$ as usual.

This is known in the case $c(n) = \Lambda(n)$ by the
Siegel--Walfisz Theorem so Corollary \ref{cor:primes} follows
immediately from Corollary \ref{cor:SWsequence} below. 

\begin{cor}\label{cor:SWsequence}
Let $c_1(n)$ and $c_2(n)$ be arithmetic functions supported on $n\leq
x$ with $c_1(n)$ satisfying the Siegel--Walfisz condition given by
\eqref{eqn:SW1}.  For $(a, q) = 1$, let  
\begin{equation}
S(x; a, q) = \sum_{\substack{m, n \leq x \\ (mn,q)=1\\ m \equiv an \modd{q}}} c_1(m)c_2(n),
\end{equation}and let
\begin{equation}
S(x; q) = \frac{1}{\phi(q)} \left(\sum_{\substack{m \le
      x\\ (m, q) = 1}} c_1(m)\right)\left(\sum_{\substack{n \le
      x\\ (n, q) = 1}} c_2(n)\right). 
\end{equation}
Then for any $A>0$, there exists $B = B(A)$ such that for $Q \leq x^2
(\log x)^{-B}$ 
\begin{equation}\label{eqn:cor1}
\sum_{q\leq Q}\; \sumstar_{a \modd{q}} \left|S(x; a, q) - S(x;
  q)\right|^2 \ll \frac{x^2}{\log^A x} \|c_1\tau\|^2\|c_2\tau\|^2.
\end{equation}
\end{cor}
Corollary \ref{cor:SWsequence} follows from the more technical Theorem
\ref{thm:2} stated in Section \ref{section:BDH} which applies to
general sequences not necessarily satisfying the Siegel--Walfisz
condition.   
 
After using Corollary \ref{cor:SWsequence} to extract main terms in
our bilinear sum, we still need to estimate the sum of these main
terms.  In Friedlander and Iwaniec's treatment, this involves a difficult direct estimation (see Sections 10-26 in  \cite{FI}). We avail ourselves of their work in our estimates as well.

In our treatment, we compare our $a^2+p^4$ with the sequence given by $a^2 + p^2$, previously studied by Fouvry and Iwaniec \cite{FouI}, which helps to streamline our arguments.  For our result, we may also use the sequence $a^2+b^2$ where $b$ has no small prime factors below $(\log x)^A$, but using the result of Fouvry and Iwaniec is convenient and elegant.  We now move to more precise definitions. 

Since the number of $n = a^2 + p^4 \leq x$ with $p< x^{1/4}/\log^2 x$
is bounded by 
\begin{equation}
\sum_{p<\frac{x^{1/4}}{\log^2 x}} \;\;\sum_{a\leq \sqrt{x - p^4}} 1 \ll
\frac{x^{3/4}}{\log^3 x}, 
\end{equation}we may assume that $p \ge x^{1/4}/\log^2 x$.  
Let $x^{1/2}/\log^4 x \leq X \le x^{1/2}$ and $\eta = (\log x)^{-1}$, and
fix an
interval $I = (X, X(1+\eta)]$.  In our treatment of the bilinear
sum, we need the two sequences to behave alike even when restricted to
small sets, so that it becomes necessary to introduce proper
weights. Define the sequences $\A = \{a(n)\}$ and $\B = \{b(n)\}$ by 
\begin{equation}\label{eqn:Adef}
a(n) = \begin{cases}
\sum_{\substack{n = a^2+p^4\\p^2 \in I\\(a, p) = 1}} 2p\log p &\textup{ if $n\leq x$}\\ 
0 &\textup{ otherwise,}
\end{cases}
\end{equation}and
\begin{equation}
b(n) = \begin{cases}
\sum_{\substack{n = a^2 + p^2\\ p\in I\\(a, p) = 1}} \log p &\textup{ if $n\leq x$}\\ 
0 &\textup{ otherwise.}
\end{cases}
\end{equation}

Note that for any fixed natural number $d$ and any real $A>0$, we have
\begin{equation}
\sum_{d|n} a(n) = \sum_{d|n} b(n) + O_{d,A} \left(x (\log x)^{-A}\right).
\end{equation}It is this property and the choice of weights which
makes the sequence $\B$ suitable for our method.

Further, define 
\begin{equation}
\pi(\A) = \sum_p a(p),
\end{equation}and
\begin{equation}
\pi(\B) = \sum_p b(p).
\end{equation}

We claim that it suffices to show
\begin{prop} \label{prop:main}
With notation as above,
\begin{equation}
\pi(\A) - \pi(\B) \ll \frac{1}{(\log x)^{2-\epsilon}}\mu(I),
\end{equation}
where
\begin{equation}
\mu(I) = \int_I \sqrt{x - t^2} dt.
\end{equation}
\end{prop}
Note that $\mu(I) \le \eta X\sqrt{x} \le \eta x$ for all values of $X$, and
$\mu(I) \gg \eta X\sqrt{x}$ for $X \le \sqrt{x}/2$ and $\mu(I) \gg \eta^{3/2}
x$ for $\sqrt{x}/2 \le X \le \sqrt{x}$. 

We verify that our main Theorem follows from Proposition
\ref{prop:main} in the next section.  Most of this paper will be
devoted to proving the Proposition. 

We conclude this introduction by recording one convention of notation.
We shall use the familiar convention that the positive number
$\varepsilon$ may vary between occurrences.  This allows us to write
$x^{\epsilon}\log x\ll x^{\epsilon}$, for example.
\\\\
{\bf Acknowledgement.} This work was supported by EPSRC grant 
EP/K021132X/1. We also thank Pierre Le Boudec for a careful reading 
of an earlier version of this paper. This resulted in the detection of a
significant oversight, which has now been corrected.

\section{Setting up the sieve}
Let us first verify that Theorem 1 follows from Proposition
\ref{prop:main}.  Note that the condition $(a, p) = 1$ in the
definition of $b(n)$ may be removed since 
\begin{equation}
\sum_{p\in I} \log p \sum_{\substack{a\le \sqrt{x-p^2}\\p|a}} 1 \ll x^{1/2+\epsilon}.
\end{equation}  Then, by the work of Fouvry and Iwaniec \cite{FouI}
and summation by parts, we have that 
\begin{equation}
\pi(\B) = \frac{\nu\mu(I)}{\log x}
\left(1+O\bfrac{1}{\log^{1-\epsilon}x}\right), 
\end{equation}for $\nu$ as given in Theorem \ref{thm:1}.  Thus
Proposition \ref{prop:main} gives us that 
\begin{equation}\label{eqn:piA}
\pi(\A) = \frac{\nu\mu(I)}{\log x}
\left(1+O\bfrac{1}{\log^{1-\epsilon}x}\right). 
\end{equation}  Our main result then follows by partial summation,
possible since the length of $I$ is short.  To be precise, let $I_j =
(X_j, X_j(1+\eta)]$ be disjoint intervals for $1\le j\le m$ such that  
\begin{equation}
\cup_j  I_j = (Y, x^{1/2}],
\end{equation}
where $Y \gg \frac{x^{1/2}}{\log^2 x}$.  Further, let $\A_j$ be
defined as in \eqref{eqn:Adef} with $I = I_j$.  Recall that $\eta =
\frac{1}{\log x}$, and note that the number of elements $a^2+p^4 \le
x$ where $p|a$ with $p^2\in I$ is bounded by  
$$\sum_{p^2 \in I} \frac{\sqrt{x}}{p} 
\ll x^{1/2+\epsilon}.
$$  Hence,
\begin{align}
&\#\{a^2+p^4 \leq x: a^2 + p^4 \textup{ is prime and }a>0\}\\
& = \sum_j \frac{1}{\sqrt{X_j}\log X_j}\pi(\A_j)
\left(1+O\bfrac{1}{\log x}\right) + O(x^{1/2+\epsilon})\notag \\ 
& = \frac{2\nu + O\bfrac{1}{\log^{1-\epsilon}x}}{\log^2 x} \sum_j
\frac{\mu(I_j)}{\sqrt{X_j}} + O(x^{1/2+\epsilon}) \notag \\ 
& = \frac{2\nu  +O\bfrac{1}{\log^{1-\epsilon}x}}{\log^2 x}
\int_Y^{\sqrt{x}} \sqrt{x - t^2} \frac{dt}{\sqrt{t}}  +
O(x^{1/2+\epsilon})\notag \\ 
& = \frac{4\nu x^{3/4}}{\log^2 x} \int_0^1 \sqrt{1 - t^4} dt
\left(1+O\bfrac{1}{\log^{1-\epsilon}x}\right). \notag \\ 
\end{align}
Thus Theorem \ref{thm:1} follows from Proposition \ref{prop:main}.  

We now fix some basic notation.  For $\C = \{c(n)\}$ any sequence
supported on $(\sqrt{x},x]$, let 
\begin{equation}
\pi(\C) = \sum_{p\leq x} c(p),
\end{equation}
\begin{equation}
\C_d = \{c(dn): n \in \mathbb{N} \},
\end{equation}
\begin{equation}
\#\C_d = \sum_{d|n} c(n),
\end{equation}and
\begin{equation}
R_d(\C) = |\# \C_d - M_d(\C)|,
\end{equation}for some $M_d(\C)$ depending on $d$ and $\C$.  Note that
the use of $\# \C_d$ here denotes the sum of the elements in $\C_d$
rather than the number of elements in $\C_d$.

We shall prove Proposition \ref{prop:main} by applying the same
sieving procedure to both sequences $\A$ and $\B$.  As mentioned in
the Introduction, this requires a level of distribution result and an
understanding of certain bilinear forms.  We refer the reader to
Friedlander and Iwaniec's asymptotic sieve for primes \cite{FI3},
Harman's alternative sieve \cite{Har} and Heath-Brown's proof of
primes of the form $x^3+2y^3$ \cite{HB} for several perspectives on
this.  Here, we develop what we need from scratch along the lines of
\cite{HB}.  For our application, this eases some technical details
involving the bilinear sum.  We begin by stating a level of
distribution result for $\A$ and $\B$. 
\begin{prop}\label{prop:leveldis}
Let 
\begin{equation}
g(d)= \frac{\rho(d)}{d},
\end{equation}
where $\rho(d)$ denotes the number of solutions to
\begin{equation}
a^2 + 1 \equiv 0 \modd{d}.
\end{equation}Now define
\begin{equation}
M_d(\C) = g(d) \mu(I).
\end{equation}Then for any constants $A \geq 0$ and  $k \geq 0$, there
exists a constant $B = B(A, k)$ such that for $D =
\frac{x^{3/4}}{(\log x)^B}$, we have  
\begin{equation}
\sum_{d \leq D} \tau^k (d) R_d(\C) \ll \frac{x}{(\log x)^A},
\end{equation}for both $\C = \A$ and $\C = \B$.
\end{prop} For $\C = \B$, the Proposition holds in the larger range $D
\leq \frac{x}{(\log x)^B}$, but we do not require this.  Proposition
\ref{prop:leveldis} is a direct consequence of a result of Friedlander
and Iwaniec \cite{FI2} to allow for the weights we need and to include
the $\tau(d)^k$ factor.  The derivation of Proposition
\ref{prop:leveldis} is given in Section \ref{section:leveldis}. 
 
Fix for the rest of the paper $\delta = (\log x)^{\varpi-1}$ for some
small constant $\varpi>0$ ($\varpi$ must be smaller than the $\epsilon$ appearing in the statement of Theorem \ref{thm:1}), and $Y = x^{1/3+1/48}$.  In fact, 
the $\frac {1}{48}$ in the
previous definition can be replaced with any positive number less than
$\frac{1}{24}$.  The following Lemma begins our sieving procedure. 
\begin{lem}\label{lem:sieve1}
For any $x^\delta < Y < x^{1/2-\delta}$ and $\C = \A$ and $\C = \B$, we have that
\begin{equation}
\pi(\C) = S_1(\C) - S_2(\C) - S_3(\C) + O\left(\frac{\delta}{\log x}
  \mu(I)\right), 
\end{equation}where
\begin{align}
S_1(\C) &= S(\C, x^\delta)\\
S_2(\C) &= \sum_{x^\delta\le p < Y} S(\C_p, p)\\
S_3(\C) &= \sum_{Y \le p< x^{1/2-\delta}} S(\C_p, p).
\end{align}
\end{lem}

\begin{proof}
By Buchstab's identity, we have
\begin{equation}
\pi(\C) = S(\C, x^{1/2}) = S_1(\C) - S_2(\C) - S_3(\C) -
\sum_{x^{1/2-\delta}\leq p \leq x^{1/2}}S(\C_p, p). 
\end{equation}Using Selberg's upper bound sieve, we have that
\begin{align}
\sum_{x^{1/2-\delta}\leq p \leq x^{1/2}} S(\C_p, p) 
&\leq \sum_{x^{1/2-\delta}\leq p \leq x^{1/2}}S(\C_p, x^{1/10}) \notag \\
&\ll \sum_{x^{1/2-\delta}\leq p \leq x^{1/2}} \frac{\mu(I)}{p \log x} +
\sum_{x^{1/2-\delta}\leq p \leq x^{1/2}}\sum_{d \leq x^{1/5}} \tau_3(d)
R_{dp}(\C) \notag  \\ 
&\ll \frac{\delta \mu(I)}{\log x} + \frac{x}{\log^A x},
\end{align}for any $A \geq 0$ by Proposition \ref{prop:leveldis}.
\end{proof}
While $S_1(\C)$ may be handled via the Fundamental Lemma, and
$S_3(\C)$ can be readily written in terms of a bilinear form in the
right range, $S_2(\C)$ requires more attention. 

\begin{lem}\label{lem:sieve2}
With notation as in Lemma \ref{lem:sieve1} and $n_0 = \left[\frac{\log
  Y}{\delta \log x}\right]$, we have 
\begin{equation}\label{eqn:S2decomp1}
S_2(\C) = \sum_{1\leq n\leq n_0} (-1)^{n-1}(T^{(n)}(\C) - U^{(n)}(\C)),
\end{equation}where
\begin{equation}
T^{(n)}(\C) = \sum_{\substack{x^\delta \leq p_n <...<p_1<Y\\ p_1...p_n
    < Y}} S(\C_{p_1...p_{n}}, x^\delta), 
\end{equation}and 
\begin{equation}
U^{(n)}(\C) = \sum_{\substack{x^\delta\le p_{n+1} <...<p_1<Y \\
    p_1...p_n < Y \leq p_1...p_{n+1}}} S(\C_{p_1...p_{n+1}},
p_{n+1}). 
\end{equation}

\end{lem}

\begin{proof}
Let 
\begin{equation}
V^{(n)}(\C) = \sum_{\substack{x^\delta \le p_n<...<p_1<Y \\ p_1...p_n
    < Y}} S(\C_{p_1...p_n}, p_n). 
\end{equation}
The proof of \eqref{eqn:S2decomp1} follows immediately upon observing
that $S_2(\C) = V^{(1)}(\C)$ and from the identity 
\begin{equation}
V^{(n)}(\C) = T^{(n)}(\C) - U^{(n)}(\C) - V^{(n+1)}(\C).
\end{equation}
\end{proof}

By Lemmas \ref{lem:sieve1} and \ref{lem:sieve2}, in order to prove
Proposition \ref{prop:main}, it suffices to prove the following two
propositions. 

\begin{prop}\label{prop:fundlem}
Let $Q$ be a set of squarefree numbers not exceeding $Y$.  Then for any $A>0$,
\begin{equation}
|\sum_{q\in Q} S(\A_q, x^\delta) - \sum_{q\in Q} S(\B_q, x^\delta)|
\ll_A \frac{x}{\log^A x}. 
\end{equation}
\end{prop}
Note that Proposition \ref{prop:fundlem} immediately implies that
\begin{equation}
|S_1(\A) - S_1(\B)| \ll_A \frac{x}{\log^A x}.
\end{equation}
Since $n_0 \asymp 1/\delta \ll \log x$, Proposition \ref{prop:fundlem}
also implies that 
\begin{equation}
\sum_{1\leq n\leq n_0} |T^{(n)}(\A) - T^{(n)}(\B)| \ll_A \frac{x}{\log^A x}.
\end{equation}

The rest of the pieces from the decompositions in Lemmas
\ref{lem:sieve1} and \ref{lem:sieve2} are handled below. 
\begin{prop} \label{prop:bilinear1}
For any $A>0$ and $n\geq 3$,
\begin{align}
|S_3(\A) - S_3(\B)| &\ll \frac{x}{\log^A x}, \textup{ and}\\
|U^{(n)}(\A)-U^{(n)}(\B)| &\ll \frac{x}{\log^A x}.
\end{align}For $n\leq 2$,
\begin{equation}
|U^{(n)}(\A)-U^{(n)}(\B)| \ll \frac{\delta }{\log x}\mu(I).
\end{equation}

\end{prop}

Propositions \ref{prop:leveldis} and \ref{prop:fundlem} are proven in
Sections \ref{section:leveldis} and \ref{section:fundlem}
respectively.  We reduce the proof of Proposition \ref{prop:bilinear1}
to a related statement about bilinear sums in Section
\ref{section:bilinear2}, which is in turn proven in a number of stages in
Sections
\ref{section:bilinearpf1} to \ref{sec:smalllarged}.

\section{Level of distribution} \label{section:leveldis}
Here, we prove Proposition \ref{prop:leveldis}.  Define the sequences
$\tilde \A$ and $\tilde \B$ by 
\begin{equation}
\tilde a(n) = \begin{cases}
\sum_{\substack{n = a^2+p^4\\p^2 \in I\\p\nmid a}} \frac{p \log
  p}{x^{1/4}\log x} &\textup{ if $\sqrt x<n\leq x$}\\ 
0 &\textup{ otherwise,}
\end{cases}
\end{equation}and
\begin{equation}
\tilde b(n) = \begin{cases}
\sum_{\substack{n = a^2 + p^2\\ p\in I\\p \nmid a}} \frac{\log p}{\log
  x} &\textup{ if $\sqrt x<n\leq x$}\\ 
0 &\textup{ otherwise.}
\end{cases}
\end{equation}
Friedlander and Iwaniec's main result in \cite{FI2} gives that for
\begin{equation}
M_d(\tilde \A) = g(d) \sum_{p^2\in I} \frac{p\log p}{x^{1/4}\log
  x}\frac{p-1}{p}\sqrt{x-p^4}, 
\end{equation}and
\begin{equation}
M_d(\tilde \B) = g(d) \sum_{p\in I} \frac{\log p}{\log x} \frac{p-1}{p}
\sqrt{x - p^2}, 
\end{equation}
and any $A>0$, there exists $B = B(A)$ such that
\begin{equation}
\sum_{d<D} R_d(\tilde \A) \ll_A \frac{x^{3/4}}{\log^A x}
\end{equation}and
\begin{equation}
\sum_{d<D} R_d(\tilde \B) \ll_A \frac{x}{\log^A x},
\end{equation}for $D = \frac{x^{3/4}}{\log^B x}$.  This immediately
implies the corresponding results for $\A$ and $\B$.  Specifically,
the following Lemma holds. 

\begin{lem}\label{lem:leveldis1}
For any $A>0$, there exists $B = B(A)$ such that
\begin{equation}
\sum_{d < D} R_d(\C) \ll \frac{x}{(\log x)^{A}}
\end{equation}for both $\C = \A$ and $\C = \B$, where $D =
\frac{x^{3/4}}{\log^B x}$. 
\end{lem}

We want to prove a version of the above lemma which includes a
$\tau(d)^k$ term.  In order to develop this, we first state an
elementary Lemma. 

\begin{lem}\label{lem:tau}
For any $n, k \geq 1$, there exists a divisor $d|n$ such that $d \le
n^{1/2^k}$ and such that 
\begin{equation}
\tau(n) \le 2^{2^k-1}\tau(d)^{2^k}.
\end{equation}
\end{lem}
\begin{proof}
Assume $n>1$.  We prove the result for $k=1$, the rest following by
induction on $k$.  Let $d|n$ with $d\leq \sqrt{n}$ be such that
$\tau(d)$ is maximal.  Write $n = d d'$, and note that $d' > 1$.  By
maximality, any prime divisor $p|d'$ satisfies $pd>\sqrt{n}$ so that
$d'/p \leq \sqrt{n}$.   Again, by maximality, $\tau(d'/p) \leq
\tau(d)$.  Thus we have that $\tau(n) \leq \tau(d)\tau(d'/p) \tau(p)
\leq 2 \tau(d)^2$. 
\end{proof}

Now we prove the following trivial bound.

\begin{lem}\label{lem:taudCd}
Recall that $D = \frac{x^{3/4}}{\log^B x}$.  For $\C = \A$ and $\C = \B$, we have
\begin{equation}
\sum_{d<D} \tau(d)^k \# \C_d \ll x (\log x)^{2^{2k+3}}
\end{equation}
\end{lem}
\begin{proof}
We prove the result for $\C = \A$, the proof for $\C = \B$ being
essentially the same.  By Lemma \ref{lem:tau}, 
\begin{align*}
\sum_{d<D} \tau(d)^k \# \A_d 
&\le \sum_n \tau(n)^{k+1} a(n)\\
&\ll \sum_{d \le \sqrt{x}} \tau(d)^{2k+2} \# \A_d.
\end{align*}
For $d\le \sqrt{x}$,
\begin{align*}
\# \A_d &\ll \sum_{p^2\in I} p\log p \sum_{\substack{a<\sqrt{x}\\a^2
    \equiv -p^4 \modd{d}\\(a, p)=1}} 1\\ 
& \ll g(d) \sqrt{x}\sum_{p^2\in I} p\log p \\
& \ll g(d) x.
\end{align*}
Using the bound $g(d) \leq \frac{\tau(d)}{d}$, we have
\begin{align*}
\sum_{d \le \sqrt{x}} \tau(d)^{2k+2} \# \A_d
\ll x \sum_{d \le \sqrt{x}} \frac{\tau(d)^{2k+3}}{d}
\ll x (\log x)^{2^{2k+3}}.
\end{align*}
\end{proof}

Now, we are ready to prove our Proposition \ref{prop:leveldis}.
By Cauchy--Schwarz, we have
\begin{align*}
\sum_{d<D} \tau(d)^k R_d(\C) 
&\ll  \left(\sum_{d<D} \tau(d)^{2k} R_d(\C) \right)^{1/2}
\left(\sum_{d<D} R_d(\C) \right)^{1/2} \\ 
&\ll  \left(\sum_{d<D} \tau(d)^{2k} (\# \C_d + M_d(\C)) \right)^{1/2}
\left(\sum_{d<D} R_d(\C) \right)^{1/2} \\ 
&\ll  \left(x (\log x)^{2^{4k+3}} \right)^{1/2}  \bfrac{x}{(\log x)^{\tilde A}}^{1/2},
\end{align*}by Lemmas \ref{lem:leveldis1} and \ref{lem:taudCd}.  This
concludes the proof since we may make $\tilde A$ as large as we like
by choosing $B$ suitably large.

\section{Application of the Fundamental Lemma}
\label{section:fundlem} 
We now prove Proposition \ref{prop:fundlem}.  Recall that we are
interested in studying 
\begin{equation}
\sum_{q\in Q} S(\C_q, x^\delta),
\end{equation}for $\C = \A$ or $\C = \B$, $\delta = (\log
x)^{\varpi-1}$,  and $Q$ a set of square-free numbers not exceeding
$Y$.  The level of distribution provided by Proposition
\ref{prop:leveldis} is sufficient to derive the correct asymptotic for
this quantity, with a small error term.  To be precise, we apply an
upper and lower bound sieve of level of distribution $x^{\frac{1}{4}}$
so that the sifting variable $s = \frac{1}{4 \delta}$.  For $z\geq 1$,
we use the usual notation 
\begin{equation}
V(z) = \prod_{p < z} \left(1- g(p)\right).
\end{equation}

By the Fundamental Lemma (see e.g. Corollary 6.10 in \cite{FI}) and
Proposition \ref{prop:leveldis}, we have  
\begin{align*}
\sum_{q\in Q} S(\C_q, x^\delta) 
&= V(x^\delta) \sum_{q\in Q} g(q) \mu(I)  \left(1+O\left(\exp\left(-(4
      \delta)^{-1} \right)\right)\right) + O\left(\sum_{q \in Q}
  \sum_{d < x^{1/4}} R_{dq}(\C) \right)\\ 
&= V(x^\delta) \sum_{q\in Q} g(q)  \mu(I)
\left(1+O\left(\frac{1}{\log^A x}\right)\right) +
O\left(\sum_{d<x^{3/4 - 1/8}} \tau(d) R_d(\C)\right)\\ 
&=V(x^\delta) \sum_{q\in Q} g(q)  \mu(I)
\left(1+O\left(\frac{1}{\log^A x}\right)\right) + O(x\log^{-A}x) 
\end{align*}for any $A>0$.  
Note that the last line is independent of whether $\C = \A$ or $\C = \B$.  Thus
\begin{align*}
\sum_{q\in Q} S(\A_q, x^\delta)  - \sum_{q\in Q} S(\B_q, x^\delta) 
&\ll \frac{1}{\log^A x} \mu(I) \sum_{q\in Q} g(q) + x\log^{-A}x\\
&\ll  x\log^{-A+2}x,
\end{align*}upon noting that $g(q) \ll \frac{\tau(q)}{q}$.  

\section{Reduction of Proposition \ref{prop:bilinear1} 
to a bilinear  form bound} \label{section:bilinear2} 
We first rewrite $U^{(1)}$ and $U^{(2)}$ into a more convenient form.
\begin{lem}
For $\C = \A$ and $\C = \B$, and $U^{(j)}$ as defined in Lemma \ref{lem:sieve2},
\begin{align} \label{eqn:U1decomp}
U^{(1)}(\C) &
= \sum_{\substack{x^\delta \leq p_2 < p_1 < Y \\ Y \leq p_1p_2 <
    x^{1/2-\delta}}} S(\C_{p_1p_2}, p_2) + \sum_{\substack{x^\delta
    \leq p_2 < p_1 < Y \\ p_1p_2 \geq x^{1/2+\delta}}} S(\C_{p_1p_2},
p_2) + O\bfrac{\delta \mu(I)}{\log x} \notag \\ 
&=: U_1^{(1)}(\C) + U_2^{(1)}(\C) + O\left(\frac{\delta }{\log x}\mu(I)\right),
\end{align}and
\begin{align}\label{eqn:U2decomp}
U^{(2)}(\C) &= \sum_{\substack{x^\delta \leq p_3<...< p_1 < Y \\
    p_1p_2 < Y \leq p_1p_2p_3 < x^{1/2-\delta}}} S(\C_{p_1p_2p_3},
p_3) +\sum_{\substack{x^\delta \leq p_3<...< p_1 < Y \\ p_1p_2 < Y
    \leq p_1p_2p_3 \\ p_1p_2p_3 \geq x^{1/2+\delta}}}
S(\C_{p_1p_2p_3}, p_3)+ O\bfrac{\delta  \mu(I)}{\log x} \notag \\ 
&=: U_1^{(2)}(\C) + U_2^{(2)}(\C)+ O\bfrac{\delta \mu(I)}{\log x}.
\end{align}
\end{lem}
\begin{proof}
To prove \eqref{eqn:U1decomp}, it suffices to show that
\begin{equation}\label{eqn:U1decomppf1}
\sum_{\substack{x^\delta \leq p_2 < p_1 < Y \\ x^{1/2 - \delta}
    \leq p_1p_2 < x^{1/2+\delta}}} S(\C_{p_1p_2}, p_2) \ll \frac{\delta
  \mu(I)}{\log x}. 
\end{equation}In the sum above, $p_2 \geq x^{1/2 - \delta}/p_1 > x^{1/2 -
  \delta}/Y > x^{1/10}$, so that by Selberg's upper bound sieve, and
Proposition \ref{prop:leveldis}, the left hand side of
\eqref{eqn:U1decomppf1} is bounded by 
\begin{align}
\sum_{\substack{x^\delta \leq p_2 < p_1 < Y \\ x^{1/2 - \delta}
    \leq p_1p_2 < x^{1/2+\delta}}} S(\C_{p_1p_2}, x^{1/10})  
&\ll  \frac{\mu(I)}{\log x} \sum_{\substack{ x^{1/10} < p_2 < p_1 < Y
    \\ x^{1/2 - \delta} \leq p_1p_2 < x^{1/2+\delta}}} \frac{1}{p_1p_2}
\notag \\ 
&\ll \frac{\delta}{\log x}\mu(I).
\end{align}

Similarly, to prove \eqref{eqn:U2decomp}, it suffices to show that
\begin{align} \label{eqn:U2decomppf1}
\sum_{\substack{x^\delta \leq p_3 < p_2 < p_1 < Y \\ p_1p_2 < Y \\
    x^{1/2 - \delta} \leq p_1p_2p_3 < x^{1/2+\delta}}} S(\C_{p_1p_2p_3},
p_3) \ll \frac{\delta }{\log x}\mu(I). 
\end{align}In the sum above, $p_3 > x^{1/2 - \delta}/Y > x^{1/10}$, so
by Selberg's upper bound sieve, and Proposition \ref{prop:leveldis},
the quantity on the left hand side of \eqref{eqn:U2decomppf1} is
bounded by 
\begin{align}
\sum_{\substack{x^\delta \leq p_3 < p_2 < p_1 < Y \\ p_1p_2 < Y \\
    x^{1/2 - \delta} \leq p_1p_2p_3 < x^{1/2+\delta}}} S(\C_{p_1p_2p_3},
x^{1/10})  
&\ll \frac{\mu(I)}{\log x} \sum_{\substack{x^{1/10}< p_3 < p_2 < p_1 <
    Y \\ x^{1/2 - \delta} \leq p_1p_2p_3 < x^{1/2+\delta}}}
\frac{1}{p_1p_2p_3} \notag \\ 
&\ll \frac{\delta}{\log x}\mu(I).
\end{align}
\end{proof}

In
order to simplify the conditions on the primes in our sieving functions the
following lemma will be useful.
\begin{lem}\label{lem:close}
Let $x^{-\de}\le\kappa\le 1$.  Then for any $P_1,P_2\in [x^{\de},x^{1/3}]$
we have
\[\sum_{P_1\le p_1\le (1+\kappa)P_1}\sum_{P_2\le p_2\le (1+\kappa)P_2}
\sum_{n\equiv 0\modd{p_1p_2}}c(n)\tau(n)\ll \kappa^2 x(\log x)^{2^{17}}
+ \frac{x}{(\log x)^A}\]
for any $A>0$.
\end{lem}
\begin{proof}
By Lemma \ref{lem:tau}, $n$ has at least one divisor $d\le n^{1/16}$
such that $\tau(n)\ll\tau(d)^{16}$.  Thus according to 
Proposition \ref{prop:leveldis} we obtain
\begin{eqnarray*}
\lefteqn{\sum_{P_1\le p_1\le (1+\kappa)P_1}\sum_{P_2\le p_2\le (1+\kappa)P_2}
\sum_{n\equiv 0\modd{p_1p_2}}c(n)\tau(n)}\\
&\ll& \sum_{P_1\le p_1\le (1+\kappa)P_1}\sum_{P_2\le p_2\le (1+\kappa)P_2}
\sum_{d\le x^{1/16}}\tau(d)^{16}\sum_{n\equiv 0\modd{dp_1p_2}}c(n)\\
& \ll &
\sum_{P_1\le p_1\le (1+\kappa)P_1}\sum_{P_2\le p_2\le (1+\kappa)P_2}
\sum_{d\le x^{1/16}}\tau(d)^{16}(M_{dp_1 p_2}(\C)+R_{dp_1 p_2}(\C))\\
&\ll & \sum_{P_1\le p_1\le (1+\kappa)P_1}\sum_{P_2\le p_2\le (1+\kappa)P_2}
\sum_{d\le x^{1/16}}\tau(d)^{17}\frac{x}{dp_1p_2} +\frac{x}{(\log x)^A}.
\end{eqnarray*}
To complete the proof we merely observe that if $P=P_1$ or $P_2$ then
\[\sum_{P\le p\le (1+\kappa)P}\frac{1}{p}\ll \kappa.\]
\end{proof}

For $k\geq 3$, in the sum in $U^{(k)}(\C)$, we have
\begin{equation}
Y \leq p_1...p_{k+1} < (p_1...p_k)^{\frac{k+1}{k}} \leq Y^{4/3} < x^{1/2 - \delta}.
\end{equation}
Thus if we define
\[U^{(k)}_*(\C)=\sum_{\substack{x^\delta\leq p_{k+1}<\ldots<p_1<Y \\
p_1\ldots p_k<Y\leq p_1\ldots p_{k+1}< x^{1/2-\delta}}} 
S(\C_{p_1\ldots p_{k+1}},p_{k+1}),\]
then we will have 
\[S_3(\C)=U^{(0)}_*(\C),\;\;U_1^{(1)}(\C)=U^{(1)}_*(\C),\;\;
U^{(2)}_1(\C)=U^{(2)}_*(\C),\]
and
\[U^{(k)}(\C)=U^{(k)}_*(\C)\;\;\mbox{for}\;\; k\ge 3.\]

If $p\in J=[V,(1+\kappa)V)$ and an integer $n$ is counted by
$S(\C_{pq},V)$ but not by $S(\C_{pq},p)$, then $n$ has at
least two prime factors in $J$.  In our application we have 
$V\le x^{1/2-\delta}$ and $n\geq x(\log x)^{-8}$.  Thus $n$ will have 
at least one further prime factor.  Thus $V^3\le n\le x$ in this 
situation. A given integer $n$ may be counted in many ways by 
$U^{(k)}_*(\C)$. However the number of ways is at most the number of
choices for $p_{k+1}<\ldots<p_1$ all dividing $n$. This will be
\[\left(\begin{array}{c} \omega(n) \\ k+1\end{array}\right)
\leq 2^{\omega(n)}\leq \tau(n).\]
The total contribution from such
integers $n$ is therefore bounded as in Lemma \ref{lem:close}. 
Now, let  
\begin{equation}
J(r) = [V_r,V_{r+1})=[x^{\delta}(1+\kappa)^r,
x^{\delta}(1+\kappa)^{r+1}),\;\;(r\ge 0)
\end{equation}
and let $R \ll \kappa^{-1}\log x$ be such that
$x^{\delta}(1+\kappa)^R>x$. We then see that
\begin{align}\label{eqn:U*}
U^{(k)}_*(\C)&=\sum_{0\le r\le R}\sum_{p\in J(r)}\sum_{\substack{p<p_k<\ldots<p_1<Y \\
p_1\ldots p_k<Y\leq p_1\ldots p_kp< x^{1/2-\delta}}} 
S(\C_{p_1\ldots p_k p},V_r)\\
&\hspace{2cm}+O(\kappa x(\log x)^{1+2^{17}})
+O\left(\kappa^{-1}\frac{x}{\log^{A-1} x}\right).
\end{align}
This procedure enables us to reduce considerations to a bilinear sum.
Indeed we have
\begin{align}\label{eqn:bilin}
\sum_{p\in J(r)}\sum_{\substack{p<p_k<\ldots<p_1<Y \\
p_1\ldots p_k<Y\leq p_1\ldots p_kp< x^{1/2-\delta}}} 
S(\C_{p_1\ldots p_k p},V_r)=
\sum_{m,n}\alpha^{(r)}_m \beta^{(r)}_n c(mn)
\end{align}
where $\alpha^{(r)}_m$ is the characteristic function for the integers
$m$ all of whose prime factors are at least $V_r$, and $\beta^{(r)}_n$ 
is the characteristic function for integers $n=p_1\ldots p_k p$ with
\[p\in J(r),\;\;\; p<p_k<\ldots<p_1<Y\;\;\;\mbox{and}\;\;\;
p_1\ldots p_k<Y\leq p_1\ldots p_kp< x^{1/2-\delta}.\]
Note that $\beta^{(r)}_n$ is
supported on integers $n\in[Y,x^{1/2-\de})$.

The procedure for $U_2^{(1)}(\C)$ and $U_2^{(2)}(\C)$ will be somewhat
different.  As before we may use Lemma \ref{lem:close} to
replace $S(\C_{p_1p_2},p_2)$ in $U_2^{(1)}(\C)$ by
$S(\C_{p_1p_2},V_r)$, when $p_2\in J(r)$.  For example, this yields
\[U_2^{(1)}(\C)=\sum_{0\le r\le R}\sum_{p_2\in J(r)}
\sum_{\substack{p_1\geq x^{1/2+\de}/p_2\\ p_2 < p_1 < Y}} S(\C_{p_1p_2},V_r)
+O(\kappa x(\log x)^{1+2^{17}})+O\left(\kappa^{-1}\frac{x}{\log^{A-1} x}\right).\]
The sum on the right can be expressed as
\[\sum_{0\le r\le R}\sum_{m,n}\alpha^{(r)}_m \beta^{(r)}_n c(mn),\]
where we now take $\alpha^{(r)}_m$ as the characteristic function for
numbers $m=p_1p_2$ with $p_2\in J(r)$, $p_2 < p_1 < Y$, and
$p_1p_2\geq x^{1/2+\de}$, and $\beta^{(r)}_n$ as the characteristic
function for numbers $n$ all of whose prime factors are at least $V_r$.
Since $c(n)$ is supported in 
\[(X^2,x]\subseteq(x(\log x)^{-8},x]\]
we may assume that $\beta^{(r)}_n$ is supported in 
\[\big(x(\log x)^{-8}Y^{-2},x^{1/2-\de}\big]\subseteq
\big(x^{1/4+1/48},x^{1/2-\de}\big]\]
say.  This is satisfactory for our purposes.
  
We may handle $U^{(2)}_2(\C)$ in a precisely analogous fashion. On
choosing $\kappa=(\log x)^{-A/2}$ we find that each of $S_3(\C)$,
$U^{(1)}_1(\C)$, $U^{(1)}_2(\C)$, $U^{(2)}_1(\C)$, $U^{(2)}_2(\C)$ and 
$U^{(k)}(\C)$ (for $k\ge 3$), can
be expressed as a sum of $O(R)$ bilinear sums as in \eqref{eqn:bilin},
together with an error term $O(x(\log x)^{1+2^{17}-A/2})$.  Thus it will
suffice to prove the following result.

\begin{prop}\label{prop:bilinear2}
Fix any $\xi > 0$ and suppose $x^{1/4 + \xi} \le N \le
x^{1/2-\delta}$.  Then, for coefficients $\alpha_m$ and $\beta_n$
as above, we have 
\begin{equation}
\sum_{N < n \leq 2N} \sum_{\substack{m < x/N }} \alpha_m\beta_n 
\big(a(mn) -b(mn)\big) \ll_{A,\xi} \frac{x}{\log^A x}, 
\end{equation}for any $A>0$.
\end{prop}

We note for future reference that $\alpha_m$ and $\beta_n$ are
supported on integers all of whose prime factors are at least
$x^{\delta}$.  In particular they vanish unless $m$ and $n$ are odd.
We also note that $|\alpha_m|,|\beta_n|\le 1$ for all $m,n$.

\section{The bilinear form over Gaussian
  integers} \label{section:bilinearpf1} 
Our purpose is to prove Proposition \ref{prop:bilinear2}.  The expression
$a^2 + p^4$ is a special value of the norm form of the Gaussian
integers, and we now take advantage of that structure.   

For $w, z \in \mathbb{Z}[i]$, let $N(w)$ denote the usual Gaussian norm and 
\begin{equation}
S_1(z, w) = \sum_{\substack{p^2 \in I\\ \tRe \bar w z = p^2}} 2p \log p,
\end{equation}and
\begin{equation}
S_2(z, w) = \sum_{\substack{p \in I\\ \tRe \bar w z = p}} \log p.
\end{equation}Note that both sums are either empty or contain only one
term.  We would now like to convert the sum over $m$ and $n$ present
in Proposition \ref{prop:bilinear2} to a sum over Gaussian integers.
We shall call $\gamma \in \mathbb{Z}[i]$ primitive if $\gamma$ is not
divisible by any rational prime. 

\begin{lem}\label{lem:elefac}
Let $\gamma \in \mathbb Z[i]$ be primitive and coprime to $2$, and let
$m$ be a positive
integer such that $m|N(\gamma)$.  Then there exist exactly four
associate choices for $\lambda \in \mathbb Z [i]$ such that
$\lambda|\gamma$ and $N(\lambda) = m$. Of these exactly one has
$\tReb(\lambda)$ positive and odd.
\end{lem}
\begin{proof}
Suppose the ideal $(\gamma)$ factors as
\[(\gamma)=P_1^{e_1}\ldots P_k^{e_k}.\]
Since $\gamma$ is primitive and coprime to 2, we have
$P_i\not=\overline{P_j}$ for every pair $i,j$. Moreover $N(P_i)$ will
be a rational prime $p_i$, and we will have
\[m=p_1^{f_1}\ldots p_k^{f_k},\]
with exponents $f_i\le e_i$.  It is then clear that $(\lambda)$ must
be
\[(\lambda)=P_1^{f_1}\ldots P_k^{f_k},\]
and the result follows.
\end{proof}

By Lemma \ref{lem:elefac},
\begin{equation}
a(mn) = \frac 12 \sum_{N(w) = m} \sum_{\substack{N(z) = n\\ \bar w z
    \textup{ primitive}}} S_1(z, w), 
\end{equation}and
\begin{equation}
b(mn) = \frac 12 \sum_{N(w) = m} \sum_{\substack{N(z) = n\\ \bar w z
    \textup{ primitive}}} S_2(z, w), 
\end{equation}
where we restrict $z$ to have $\tReb(z)$ positive and odd in both sums.
Note that the double sum counts pairs with $\bar w z=p^2+ia$ (or
$p+ia$) with no restriction on the sign of $a$.  In our original definition of $a(n)$ and $b(n)$, we have the condition $a>0$, and this is accounted for by the factor of $\tfrac12$. 

We let $\beta_z = \beta_{N(z)}$ and $\alpha_w = \alpha_{N(w)}$.  It
now suffices to show that 
\begin{equation}\label{eqn:bilinear2}
\sum_z \sum_{\substack{w \\ \bar w z \textup{ primitive}}} \beta_z
\alpha_w \left(S_1(z, w) - S_2(z, w)\right) \ll_A \frac{x}{\log^A x}, 
\end{equation}for any $A>0$, and for coefficients $\beta_z$ and
$\alpha_w$ satisfying $|\beta_z| \leq 1$ and $|\alpha_w| \leq 1$.
Further, we may assume that $\beta_z$ is supported on primitive $z$ satisfying
$2\nmid\tReb(z)>0$ and $N \leq
N(z) < 2N$, while $\alpha_w$ is supported on primitive $w$ such that
$N(w) \leq M := x/N$.  Note that $N < M$ since $N \le x^{1/2 -
  \delta}$.  We also remark that $\beta_z$ is supported on values with
$N(z)$ free of small prime factors.  Hence $N(z)$ is odd, and since
$\tReb(z)$ is also odd we must have $z\equiv 1\modd{2}$.
 
We first remove the primitivity condition on $\bar w z$ with
negligible error.  Indeed, the contribution of $S_1(z, w)$ for
imprimitive $\bar w z$ of the form 
\begin{equation}
\bar w z = p^2 + ia
\end{equation} must have $p|a$ so that
\begin{equation}
\bar w z = p^2 + ibp,
\end{equation}for $p^2 \in I$, whence $b \le \frac{x^{1/2}}{p} \le
x^{1/4} \log^2 x$.  Hence, there are at most $x^{1/4+\epsilon}$
choices for $b$, and thus at most $x^{1/4+\epsilon}$ choices for $w$
and $z$, given $p$.  These are counted with weight $2p\log p \ll
x^{1/4+\epsilon}$.  Thus, the total contribution is bounded by 
\begin{equation}
\sum_{p \le x^{1/4}} x^{1/2+\epsilon} \ll x^{3/4+\epsilon} \ll \frac{x}{\log^A x}.
\end{equation}
The contribution from $S_2(z, w)$ for imprimitive $\bar w z$ is
bounded similarly.

Let $\theta(z) = \arg z \in [0, 2\pi)$ and
\begin{equation}
\R = \R(A) = \{z \in \mathbb{Z}[i]: N\le N(z) < 2N, |\theta(z) -
k\pi/2| \le (\log x)^{-A} \textup{ for } k\in \mathbb Z \}. 
\end{equation}We now note that we may discard the part of the sum
\eqref{eqn:bilinear2} with $z\in \R$.  

\begin{lem}\label{lem:qzfixed}
Suppose that both $z$ and $q$ are fixed.  Then the number of possible
$w$ with $q=\tRe \bar w z$ is 
\begin{equation}
\ll \frac{\sqrt{M}}{\sqrt{N}}.
\end{equation}
\end{lem}
\begin{proof}
Let 
\begin{align*}
z &= s+it \\
w &= u+iv,
\end{align*}so that
\begin{equation}\label{eqn:q}
q := \tRe \bar w z = us + vt
\end{equation}
We have either $|s| \gg \sqrt{N}$ or $|t|\gg \sqrt{N}$.  We deal with the
case $|s| \gg \sqrt{N}$, the other case being similar.  Since $z$ is
primitive, $(s, t) = 1$ so we may write 
\begin{equation}
v \equiv \bar t q \modd{s}.
\end{equation} Thus, there are $\ll \sqrt{\frac{M}{N}}$ choices for
$v$.  Once $v$ is fixed, $u$ is uniquely determined by \eqref{eqn:q}. 
\end{proof}

\begin{lem}\label{lem:zcutR}
\begin{equation}
\sum_{z\in \R} \sum_w \beta_z \alpha_w S_j(z, w) \ll x(\log x)^{-A}
\end{equation}
for $j = 1, 2$.
\end{lem}
\begin{proof}
We apply Lemma \ref{lem:qzfixed} to get
\begin{align*}
\sum_{\substack{z\in \R}} \sum_w \beta_z \alpha_w S_1(z, w)
&\ll \sum_{\substack{z\in \R\\z \textup{ primitive}}} \sum_{p^2 \in I} p \log p \sum_{\substack{w \\ \tRe \bar w z= p^2}} 1 \\
&\ll \sqrt{\frac MN} \sum_{p^2 \in I} p \log p \sum_{z\in \R} 1 \\
&\ll \sqrt{\frac MN} N (\log x)^{-A} \sum_{p^2 \in I} p\log p \\
&\ll x (\log x)^{-A}.
\end{align*}

In the case of $S_2(z, w)$, the sum is simpler and we get
\begin{align*}
\sum_{\substack{z\in \R}} \sum_w \beta_z \alpha_w S_2(z, w)
&\ll \sqrt{\frac MN}  \sum_{p \in I} \log p \sum_{z\in \R} 1\\
&\ll x (\log x)^{-A}.
\end{align*}
\end{proof}

In the sequel, let
$\sumb$ denote a sum over primitive $z \not \in \R$ for which
$\tReb(z)$ is positive and $z\equiv 1\modd{2}$.   
Then Cauchy--Schwarz gives that
\begin{equation}
\left(\sum_w \alpha_w \sumb_z \beta_z (S_1(z, w) - S_2(z, w))\right)^2
\le \sum_w \alpha_w^2  \sum_w \left(\sumb_z \beta_z (S_1(z, w) -
  S_2(z, w))\right)^2, 
\end{equation}where we now extend the sum over $w$ over all Gaussian
integers $w$ satisfying $N(w) \leq x/N$, possible by positivity. 
We then see that it suffices to show that
\begin{align}\label{eqn:bdd2}
&\sumb_{z_1, z_2} \beta_{z_1} \beta_{z_2} \sum_w
\left(S_1(z_1, w) - S_2(z_1, w)\right) \left(S_1(z_2, w) - S_2(z_2, w)
\right) \ll \frac{x N}{\log^A x}
\end{align}
for any $A>0$.

\begin{lem}
The contribution of the diagonal term $z_1 = z_2$ in \eqref{eqn:bdd2}
is at most 
\begin{equation}
\sum_z \beta_{z}^2\sum_w \left( S_1(z, w)^2 - 2 S_1(z, w)S_2(z, w) +
  S_2(z, w)^2\right) \ll x^{1-\xi/2}N.  
\end{equation} 
\end{lem}
\begin{proof}
Since it is impossible for $\tRe \bar w z$ to be both a prime and the
square of a prime, $S_1(z, w)S_2(z, w) = 0$.  Let us record the
trivial bounds 
\begin{equation}\label{eqn:S1anx}
\sum_z \sum_w S_1(z, w) \ll \sum_{n\leq x}
\left(\sum_{\substack{a^2+p^4 =n\\p^2 \in I}}2p\log p\right) \tau(n)
\ll x^{1+\epsilon} 
\end{equation}and similarly
\begin{equation}\label{eqn:S2bnx}
\sum_z \sum_w  S_2(z, w) \ll x^{1+\epsilon}.
\end{equation}

Now
\begin{align}
\sum_z \sum_w S_1(z, w)^2 + S_2(z, w)^2 &\ll \sqrt{X}\log X \sum_z
\sum_w S_1(z, w) + \log X \sum_z \sum_w S_2(z, w) \notag \\ 
&\ll x^{5/4+\epsilon},
\end{align}by \eqref{eqn:S1anx} and \eqref{eqn:S2bnx}.  Since
$N>x^{1/4+\xi}$ this suffices
on choosing $\epsilon$ sufficiently small.
\end{proof}

Thus, in considering \eqref{eqn:bdd2}, we will assume that $z_1 \neq z_2$.
For any pair $z_1, z_2$, we let \linebreak 
$\theta = \theta(z_1, z_2) = \arg z_2 -
\arg z_1$ denote the angle between $z_1$ and $z_2$.  Moreover, we define
$\Delta = \Delta(z_1, z_2) = \tIm \bar z_1 z_2 = |z_1z_2|\sin
\theta(z_1, z_2)$.  Note that $z_1$ and $z_2$ being primitive and $z_1
\neq z_2$ implies that $\theta \neq 0$.  Further, $\tRe z_i > 0$
implies that $\theta \neq \pi$.  Hence $\Delta 
\neq 0$. Since $z_1\equiv z_2\equiv 1\modd{2}$ we
will have $2\mid \Delta$.

\begin{rem}
For ease of notation, we restrict our attention to those $z_1, z_2$
satisfying $\Delta > 0$, and henceforth assume this
condition to be included in $\sumb_{z_1,z_2}$. 
\end{rem}

If we write
\begin{equation}
\tRe \bar w z_i = q_i
\end{equation}for $i = 1, 2$, then
\begin{equation}\label{eqn:wsolve}
w = -i \Delta(z_1, z_2)^{-1} (q_1 z_2 - q_2 z_1).
\end{equation}Of course, we must have
\begin{equation}\label{eqn:modD}
q_1 z_2 \equiv q_2 z_1 \modd{\Delta}.
\end{equation}
Let $C(q_1, q_2, z_1, z_2)$ be the statement that $q_1, q_2, z_1$ and
$z_2$ satisfy \eqref{eqn:modD}.  From \eqref{eqn:wsolve} and since
$N(w) \le x/N$, we have the additional condition 
\begin{equation}\label{eqn:wcond}
|q_1z_2 - q_2 z_1| \leq\Delta(z_1, z_2)\sqrt{\frac xN}.
\end{equation}

We also wish to dispose of the case in which $\Delta$ is small.  In particular, we wish to only consider those $z_1, z_2$ such that
\begin{equation}\label{eqn:Dcond}
\Delta(z_1, z_2) > \Df_0:=N(\log x)^{-A-6}. 
\end{equation}
For brevity, let
\begin{equation}
f(q) = \begin{cases}
2p\log p &\textup{ if $q = p^2\in I$}\\
0  &\textup{ otherwise,}
\end{cases}
\end{equation}and
\begin{equation}
g(q) = \begin{cases}
\log p &\textup{ if $q = p\in I$}\\
0  &\textup{ otherwise.}
\end{cases}
\end{equation}
Set
\begin{equation}\label{eqn:hdef}
h(q) = f(q) - g(q).
\end{equation}
For any $J \subset I$, we have by the Prime Number Theorem that
\begin{equation}
\sum_{q\in J} h(q) = O\bfrac{\sqrt{x}}{\log^C x},
\end{equation}for any $C>0$.  This is a result of our choice of weights.  

The conditions \eqref{eqn:wcond} and \eqref{eqn:Dcond} are quite awkward,
so we shall remove them by dissecting our sum in \eqref{eqn:bdd2}
into smaller pieces.   
To be precise, for some constant $L$ to be determined, let
\begin{equation}
\omega_1 \asymp \omega_2 \asymp \omega := (\log x)^{-L},
\end{equation}and let $I = (X, X(1+\eta)]$ be a disjoint union of
intervals $J$ of length $\asymp X \omega_1$.  We need $\ll 1/\omega_1$
such intervals to cover $I$. 
Further, split the sum over $z_1$ and $z_2$ into regions
$\U$, where each $\U$ is of the form
\begin{equation}\label{eqn:Udef}
\U(c, \theta_0) = \U := \{z: c\sqrt{N} < |z| \leq
c(1+\omega_1)\sqrt{N}, \theta_0 < \arg(z)< \theta_0 +\omega_2\}, 
\end{equation}for fixed $1\leq c < \sqrt{2}$ and $\theta_0$.  Note
that we may chose $\omega_1$ and $\omega_2$ so that our regions $\U$
form a partition of the region $\{z: N\leq N(z)<2N, \tReb(z) > 0\}-\R$.
The number of regions needed for the sum over $z_1$ and $z_2$ is
$O(\log^{4L}x)$.  Here, we have allowed $\omega_1$ to possibly be
distinct from $\omega_2$ in order to cover our region perfectly.  They
are the same size and can frequently be replaced by $\omega$ in our
estimates. 

Now, write $\Cf_1(\U_1, \U_2, J_1, J_2)$ as the condition that all $(z_1,
z_2, q_1, q_2) \in \U_1 \times \U_2 \times J_1 \times J_2$ satisfy
\eqref{eqn:wcond} and \eqref{eqn:Dcond}.  
Also, let $\Cf_2(\U_1, \U_2, J_1, J_2)$ be the
condition that there exists some $(z_1, z_2, q_1, q_2)$ in $\U_1 \times
\U_2 \times J_1 \times J_2$ which satisfies \eqref{eqn:wcond}, and
there exists some $(z_1', z_2', q_1', q_2')$ in $\U_1 \times \U_2
\times J_1 \times J_2$ which does not satisfy \eqref{eqn:wcond}.  Finally, let $\Cf_3(\U_1, \U_2, J_1, J_2)$ be the condition that all $(z_1, z_2, q_1, q_2)$ in $\U_1 \times \U_2 \times J_1 \times J_2$ satisfy \eqref{eqn:wcond} but there exists some $(z_1, z_2, q_1, q_2)$ in $\U_1 \times \U_2 \times J_1 \times J_2$ which does not satisfy \eqref{eqn:Dcond}.

For $\U_1, \U_2, J_1, J_2$ satisfying $\Cf_1(\U_1, \U_2, J_1, J_2)$, set
\begin{equation}
T(\U_1, \U_2, J_1, J_2) = \sumb_{\substack{z_1 \in \U_1\\z_2 \in
    \U_2}}\beta_{z_1}\beta_{z_2} \sum_{\substack{q_1 \in J_1\\q_2 \in J_2\\C(q_1,
      q_2, z_1, z_2)}} h(q_1)h(q_2),
\end{equation}and otherwise set $T(\U_1, \U_2, J_1, J_2) =0$.

Further, let
\begin{equation}
T'(\U_1, \U_2, J_1, J_2) = \sumb_{\substack{z_1 \in \U_1\\z_2 \in
    \U_2}} \sum_{\substack{q_1 \in J_1\\q_2 \in J_2\\C(q_1, q_2, z_1,
    z_2)}}|h(q_1)h(q_2)|. 
\end{equation}
Then \eqref{eqn:bdd2} reduces to proving that for any constant $A>0$,
\begin{equation}\label{eqn:bddd3}
\sum_{\substack{\U_1, \U_2, J_1, J_2 \\ \Cf_1(\U_1, \U_2, J_1, J_2)}}
T(\U_1, \U_2, J_1, J_2) + \sum_{\substack{\U_1, \U_2, J_1, J_2 \\
    \Cf_2(\U_1, \U_2, J_1, J_2) \textup{ or } \Cf_3(\U_1, \U_2, J_1, J_2)}} T'(\U_1, \U_2, J_1, J_2) \ll_{A}
\frac{xN}{\log^A x}. 
\end{equation}

Since $L$ may be freely chosen, it suffices to prove the following Propositions.
\begin{prop}\label{prop:bilinear3b}
With notation as above and for $L\ge A+6$, we have
\begin{equation}\label{eqn:Tprimesum}
\sum_{\substack{\U_1, \U_2, J_1, J_2 \\ \Cf_2(\U_1, \U_2, J_1, J_2) \textup{ or } \Cf_3(\U_1, \U_2, J_1, J_2)}}
T'(\U_1, \U_2, J_1, J_2) \ll \frac{xN}{\log^{A} x}. 
\end{equation}
\end{prop}
\begin{prop}\label{prop:bilinear3a}
With notation as above and for fixed $J_1, J_2$ and for $L =6A+52$
we have that 
\begin{equation}\label{eqn:Tsum}
\sum_{\substack{\U_1, \U_2\\ \Cf_1(\U_1, \U_2, J_1, J_2)}} T(\U_1, \U_2, J_1, J_2) \ll \frac{xN}{\log^{A+2L} x}.
\end{equation}
\end{prop}

Note that \eqref{eqn:bddd3} follows from the Propositions above.

\begin{rem}
When $\Cf_1(\U_1, \U_2, J_1, J_2)$ holds and $z_i \in \U_i$, we automatically have
$\tReb(z_i)>0$, $z_i\not\in \R$ and
$\Delta(z_1, z_2) > 0$.
\end{rem}

\section{Proof of Propositions \ref{prop:bilinear3b} and 
\ref{prop:bilinear3a}; Preliminary Steps}
\label{section:bilinearpf3}
We first note that we may essentially assume that $q_1q_2$ is coprime
with $\Delta$ in the sums defining $T(\U_1, \U_2, J_1, J_2)$
and $T'(\U_1, \U_2, J_1, J_2)$.

\begin{lem}\label{lem:pcoprimeD}
We have that
\begin{equation}
\sumb_{z_1, z_2} \sum_{\substack{q_1\in J_1, q_2 \in J_2 \\C(q_1, q_2,
    z_1, z_2)\\(q_1q_2, \Delta)>1}} |h(q_1)h(q_2)| \ll N^2 \sqrt{x}
\log^3 x. 
\end{equation}
\end{lem}
\begin{proof}
Note the number of $z_1, z_2$ appearing in the sum is $O(N^2)$, 
and that $h(q_1)h(q_2) \ll \sqrt{x}\log^2 x$.  Hence, it
suffices to show that for fixed $z_1$ and $z_2$, the number of choices
for $q_1$ and $q_2$ is bounded by $O(\log x)$.   

Suppose $(q_1, \Delta) > 1$.  We have that either $q_1 = p$ or $q_1 =
p^2$ for some prime $p$, so $p|\Delta$.  The congruence 
\begin{equation}
q_1 z_2 \equiv q_2 z_1 \modd{\Delta} 
\end{equation}implies that $p|q_2$ as well, since $z_1$ is primitive.
Thus $q_2 = p$ or $q_2 = p^2$ as well.  Since the number of prime
factors of $\Delta$ is $O(\log x)$, we have that the number of choices
for $p$ is $O(\log x)$, which suffices. 
\end{proof}

If $(q_i, \Delta) =1$, \eqref{eqn:modD} is equivalent to there
existing $a \modd{\Delta}$ with $(a, \Delta) = 1$ such that 
\begin{align}
q_1 \equiv a q_2 \modd{\Delta} \notag \\
a z_2 \equiv z_1 \modd{\Delta}.
\end{align}
Then, by Lemma \ref{lem:pcoprimeD}, we may rewrite
$T(\U_1, \U_2, J_1, J_2)$ as
\begin{align}\label{eqn:Sdef2}
T(\U_1, \U_2, J_1, J_2) = 
\sum_{D \le 2N} \sumstar_{a\modd{D}} Y(a, D; h, h)Z(a, D) +
O(N^2 \sqrt{x}\log^3 x) 
\end{align}where
\begin{equation}
Z(a, D) = \sumb_{\substack{(z_1, z_2) \in \U_1 \times \U_2\\
    \Delta = D \\ az_2 \equiv z_1 \modd{D}}}\beta_{z_1}\beta_{z_2} 
\end{equation}
and 
\begin{align}
Y(a, D; h_1, h_2) = Y(a, D) =  \sum_{\substack{q_1 \in J_1, q_2 \in
    J_2\\q_1 \equiv aq_2 \modd{D}\\(q_1q_2, D)=1}} h_1(q_1)h_2(q_2). 
\end{align}

The rewriting of $T(\U_1, \U_2, J_1, J_2)$ in \eqref{eqn:Sdef2} 
separates the sum $Z(a, D)$
containing the coefficients $\beta_z$ from the congruence sum $Y(a,
D)$ involving the primes.
This key procedure has transformed the sum into the right form for us
to extract the main terms from $T(\U_1, \U_2, J_1, J_2)$ using Corollary
\ref{cor:SWsequence}.  Of course, we also need some understanding of
the behaviour of $Z(a, D)$ for which the following bounds
will suffice for the moment. 

\begin{lem} \label{lem:ZbD}
Let
\begin{equation}
\tZ(a, D) = \sumb_{\substack{(z_1, z_2) \in \U_1 \times \U_2
\\ \Delta = D\\ az_2 \equiv z_1 \modd{D}}} 1.
\end{equation}
We have
\begin{equation} \label{eqn:lemZaD0}
\sum_D\tau(D)\sumstar_{a \modd{D}} \tZ(a, D) \ll \omega^4N^2(\log x)^{16},
\end{equation}
\begin{equation} \label{eqn:lemZaD1}
\sum_{\U_1, \U_2} \sumstar_{a \modd{D}} \tZ(a, D) \ll  N,
\end{equation}
and
\begin{equation} \label{eqn:lemZaD2}
\sumstar_{a \modd{D}} \tZ(a, D)^2 \ll  (\log x)^{3} \frac{N^2}{D}\tau(D)^6.
\end{equation}
\end{lem}
\begin{proof}
We write $z_k = x_k + i y_k$ for $k=1, 2$, and assume without loss of
generality that $|x_2|$ is maximal among $|x_1|, |x_2|, |y_1|$ and
$|y_2|$.  For \eqref{eqn:lemZaD0} we apply Lemma \ref{lem:tau} with
$k=2$ to deduce that
\begin{eqnarray*}
\lefteqn{\sum_D\tau(D)\sumstar_{a \modd{D}} \tZ(a, D)}\\
& \ll&  
\sum_{d\le(2N)^{1/4}}\tau(d)^4\#\{(z_1,z_2)\in\U_1\times\U_2:
\,(x_1,y_1)=1,\,d\mid x_1y_2 - x_2 y_1\}.
\end{eqnarray*}
Since the regions $\U_i$ are contained in squares of side
$O(\omega\sqrt{N})$, and $N^{1/4}\ll\omega\sqrt{N}$ we deduce that
\begin{eqnarray*}
\lefteqn{\sum_D\tau(D)\sumstar_{a \modd{D}} \tZ(a, D)}\\
& \ll&  
\sum_{d\le(2N)^{1/4}}\tau(d)^4\frac{\omega^4N^2}{d^4}\#\{(z_1,z_2)\modd{d}:\,
(x_1,y_1,d)=1,\,d\mid x_1y_2 - x_2 y_1\}.
\end{eqnarray*}
One can easily show that if $d$ is a prime power $p^e$, then
\[\#\{(z_1,z_2)\modd{d}:\,(x_1,y_1,d)=1,\,d\mid x_1y_2 - x_2 y_1\}\le
d^3,\]
whence the same bound holds for all
$d$, and we obtain
\[\sum_D\tau(D)\sumstar_{a \modd{D}} \tZ(a, D) \ll  
\omega^4N^2\sum_{d\le(2N)^{1/4}}\frac{\tau(d)^4}{d}\ll
\omega^4N^2(\log x)^{16}\]
as required.

For \eqref{eqn:lemZaD1} we note that
the condition $\Delta=D$ implies $x_1y_2 - x_2 y_1 = D$, and since
$(x_2,y_2)=1$ we have $x_1\equiv \bar y_2 D\modd{|x_2|}$. However we
arranged that $|x_1|\le |x_2|$, so that there are at most 2
possibilities for $x_1$ once $x_2$ and $y_2$ are given. Since
$x_1,x_2$ and $y_2$ determine $y_1$ via the equation 
$x_1y_2 - x_2 y_1 = D$, the required bound $O(N)$ follows.

Finally, to prove \eqref{eqn:lemZaD2} we decompose $\tZ(a,D)$ into 4
parts according to which of $|x_1|$, $|x_2|$, $|y_1|$ and
$|y_2|$ is maximal.  Using Cauchy's inequality it then suffices to handle
the analogue of $\tZ(a,D)$ in which $|x_2|$, say, is largest. If
$(x_1,x_2,D)=k$, say, we see that the congruence $x_1y_2\equiv
D\modd{|x_2|}$ determines at most $2k$ values of $y_2$ with $|y_2|\le |x_2|$.
Now, note that if we have two solutions $x_1 \equiv a x_2 \modd{D}$
and $x_1' \equiv a x_2' \modd{D}$, then $x_1x_2' \equiv x_2x_1'
\modd{D}$.  Thus, we have 
\[\sumstar_{a \modd{D}} \tZ(a, D)^2\ll
\sum_{\substack{x_1,x_2,x_1',x_2'\\ x_1x_2'\equiv x_2x_1'\modd{D}}} 
(x_1,x_2,D)(x_1',x_2',D).\]
However $(x_1,x_2,D)(x_1',x_2',D)$ divides $(x_1x_2',x_2x_1',D^2)$.  Thus,
\begin{align*}
\sumstar_{a \modd{D}} \tZ(a, D)^2
&\ll \sum_{\substack{m, n \leq 2N\\ m \equiv n \modd{D}}} \tau(m)
\tau(n)(m,n,D^2) \\
&\le \frac 12 \sum_{\substack{m, n \leq 2N\\ m \equiv n \modd{D}}}
(\tau(m)^2(m,D^2) +  \tau(n)^2(n,D^2))\\ 
&\ll \frac{N}{D} \sum_{m \leq 2N} \tau(m)^2(m,D^2).
\end{align*}
Finally
\begin{align*}
\sum_{m \leq 2N} \tau(m)^2(m,D^2) 
&\le \sum_{d|D^2} d\sum_{\substack{m \leq 2N\\ d|m}} \tau(m)^2 \\
&\le \sum_{d|D^2} d\sum_{v \leq 2N/d} \tau(v)^2\tau(d)^2 \\
&\ll \sum_{d|D^2} d\{Nd^{-1}\log^3 N\}\tau(d)^2\\
&\ll N\tau(D^2)^3\log^3 N\\
&\ll N\tau(D)^6\log^3 N
\end{align*}
which suffices to prove \eqref{eqn:lemZaD2}.  
\end{proof}
\begin{rem}
Trivially $|Z(a, D)| \leq \tZ(a, D)$ so that Lemma
\ref{lem:ZbD} applies with $\tZ(a, D)$ replaced by $Z(a, D)$. 
\end{rem}

Now for an interval $J$ and any function $\tilde{h}$, let
\begin{align}\label{eqn:YAB}
Y(J,\tilde{h};D) = \sum_{\substack{q\in J\\ (q,D)=1}}\tilde{h}(q)
\end{align}and
\begin{align}\label{eqn:YjD1}
Y_{h_1, h_2}(D) = Y(D) &= \frac{1}{\phi(D)}Y(J_1, h_1;D)Y(J_2, h_2;D). \notag \\
\end{align}

Recall that $q_1$ and $q_2$ appearing in $Y(a, D)$ satisfy $(q_1q_2,
D) =1$.  If $h_1$ or $h_2$ is $g$, then $Y(D)$ is the expected value
of $Y(a, D)$. 

If $h_1 = h_2 = f$, note that $p_1^2 \equiv ap_2^2 \modd{D}$ implies
that $p_1 \equiv bp_2 \modd{D}$ for some $b$ such that
$a \equiv b^2 \modd{D}$.  Here, $Y(a, D) = 0$
if $a$ is not a square modulo $D$ so 
\begin{equation}
\sumstar_{a\modd{D}} Y(a, D)Z(a, D) = \sumstar_{b \modd{D}} Y_f(b, D)Z(b^2, D), 
\end{equation}
where
\begin{equation}
Y_f(b, D) = \sum_{\substack{p_1^2\in J_1\\p_2^2 \in J_2\\ p_1
    \equiv b p_2 \modd{D} \\(p_1p_2, D) = 1}} f(p_1^2)f(p_2^2).
\end{equation}
When $h_1
= h_2 = f$, $Y(D)$ is the expected value of $Y_f(b, D)$.

The following proposition makes the above discussion precise.

\begin{prop}\label{prop:EMN}
If either $h_1 = g$ or $h_2=g$, let
\begin{equation}
\E(N) = \sum_{D \leq 2N} \sumstar_{a \modd{D}} 
\left|Y(a,D;h_1,h_2) - Y_{h_1,h_2}(D) \right|\tZ(a, D).
\end{equation}Then for any constant $C>0$, 
\begin{equation}\label{eqn:Ebound}
\E(N) \ll_C \frac{xN}{\log^C x}.
\end{equation}

For $h_1 = h_2 = f$, let
\begin{equation}
\E_f(N) = \sum_{D \leq 2N} \;\sumstar_{b \modd{D}} \left|Y_f(b, D) -
Y_{f,f}(D)\right|\tZ(b^2, D). 
\end{equation}Then for any constant $C>0$, 
\begin{equation}
\E_f(N) \ll_C \frac{xN}{\log^C x}.
\end{equation}
\end{prop}

\begin{proof}
We first prove the bound \eqref{eqn:Ebound} for $\E$.  We prove the
result for $h_1 = g$ and $h_2 = f$.  The proof is similar for the
cases $h_1 = g$ and $h_2 = f$ and $h_1 = h_2 = g$.  We write  
\begin{align*}
\E(N) \leq \sum_{D \leq 2N} \sum_{\substack{q_2 \in J_2\\(q_2, D)
    =1}}f(q_2) \sumstar_{a \modd{D}}  \left| \sum_{\substack{q_1\in
      J_1\\q_1 \equiv aq_2 \modd{D}}}g(q_1) - \frac{Y(J_1, g;D)}{\phi(D)}
\right| \tZ(a, D). 
\end{align*}

By Cauchy--Schwarz and Lemma \ref{lem:ZbD}, we have that
$\E(N)\le\E_1^{1/2}\E_2^{1/2}$, where
\begin{eqnarray*}
\E_1&=&\sum_{D \leq 2N} Y(J_2,f;D)\sumstar_{a \modd{D}}\tZ(a, D)^2\\ 
&\ll &\sum_{D \leq 2N} Y(J_2,f;D)(\log x)^{3} \frac{N^2}{D}\tau(D)^6,
\end{eqnarray*}
and
\begin{eqnarray*}
\E_2&=&\sum_{D \leq 2N} \sum_{\substack{q_2 \in J_2\\(q_2,
      D)=1}} f(q_2) \sumstar_{a \modd{D}}  \left(
    \sum_{\substack{q_1\in J_1\\q_1 \equiv aq_2 \modd{D}}}g(q_1) -
    \frac{Y(J_1, g;D)}{\phi(D)} \right)^2\\
&=&\sum_{D \leq 2N} \sum_{\substack{q_2 \in J_2\\(q_2, D)=1}}
  f(q_2) \sumstar_{b \modd{D}}  \left( \sum_{\substack{q_1\in J_1\\q_1
        \equiv b \modd{D}}}g(q_1) - \frac{Y(J_1, g;D)}{\phi(D)}\right)^2.
\end{eqnarray*}
We have $Y(J_2,f,D)\ll X\ll x^{1/2}$ and
\[\sum_{D\le 2N}D^{-1}\tau(D)^6\ll (\log x)^{64},\]
whence $\E_1\ll N^2x^{1/2}(\log x)^{67}$.  Moreover the bound
$Y(J_2,f,D)\ll X\ll x^{1/2}$ shows that
\[\E_2\ll x^{1/2}\sum_{D \leq 2N}
 \sumstar_{b \modd{D}}  \left( \sum_{\substack{q_1\in J_1\\q_1
        \equiv b \modd{D}}}g(q_1) - \frac{Y(J_1, g;D)}{\phi(D)}\right)^2,\]
whence it suffices to show that
\begin{equation}
\sum_{D \leq 2N} \sumstar_{b \modd{D}}  \left( \sum_{\substack{q_1\in
      J_1\\q_1 \equiv b \modd{D}}}g(q_1) - \frac{Y(J_1,
    g; D)}{\phi(D)}\right)^2 \ll_C \frac{x}{\log^C x}, 
\end{equation}for any $C>0$.  This follows by the
Barban--Davenport--Halberstam theorem (see, e.g. Theorem 9.14 in
\cite{FIO}), and noting that $\sqrt{x} \geq X \geq \sqrt{x}/\log^4 x$
while $N \leq x^{1/2 - \delta}$. 

It is necessary to use Corollary \ref{cor:SWsequence} in the bound for
$\E_f$.  Here, Cauchy--Schwarz on $\E_f$ produces
$\E_f\le\E_1^{1/2}\E_2^{1/2}$, where now
\[\E_1=\sum_{D\le N}   \sumstar_b \tZ(b^2, D)^2 ;\;\;\mbox{and}
\;\;\;
\E_2=\sum_{D \leq 2N} \sumstar_{b \modd{D}} (Y_f(b, D) - Y_{f,f}(D))^2.\]
For $\E_1$ we note that any residue class $a$ coprime to $D$ arises
$O(\tau(D))$ times as a square, whence
\begin{eqnarray*}
\E_1&\ll & \sum_{D\le N} \tau(D) \sumstar_a \tZ(a, D)^2\\
&\ll& \sum_{D\le N} \tau(D)^7 \frac{N^2}{D} \log^3 x \\
&\ll &N^2\log^{131} x,
\end{eqnarray*}
by Lemma \ref{lem:ZbD}.
Thus, it remains to show that $\E_2\ll x^2(\log x)^{-C}$ for any
constant $C$.  But $\E_2$ is simply
\[\sum_{D \leq 2N} \sumstar_{b \modd{D}} \left|\sum_{\substack{m, n\\ m
      \equiv bn \modd{D}\\ (mn, D)=1}}c_1(m)c_2(n) - \frac{1}{\phi(D)}
  Y(J_1, f;D) Y(J_2, f;D)\right|^2 \]
where 
\begin{equation}
c_1(m) = 
\begin{cases}
2p\log p &\textup{ if $m=p$ is prime and } p^2 \in J_1, \\
0 &\textup{ otherwise,}
\end{cases}
\end{equation}and similarly
\begin{equation}
c_2(n) = 
\begin{cases}
2p\log p &\textup{ if $n=p$ is prime and } p^2 \in J_2, \\
0 &\textup{ otherwise.}
\end{cases}
\end{equation}

Since $c_1$ satisfies the Siegel--Walfisz condition, we now
apply Corollary \ref{cor:SWsequence} to complete the proof.  Note that
\begin{equation}
\|c_i\tau\|^2  \ll x^{3/4} \log^2 x,
\end{equation}
for $i=1$ or 2, and that $c_i(n)$ is supported on $n \leq x^{1/4}$ (so that
the value of $x$ appearing in the statement of Corollary
\ref{cor:SWsequence} is $x^{1/4}$ in this application).   

\end{proof}

\section{Proof of Proposition \ref{prop:bilinear3b}}
\label{section:bilinearpf3b}

We begin by observing that Lemma \ref{lem:pcoprimeD} yields
\[T'(\U_1, \U_2, J_1, J_2)\ll
\sum_{D\leq 2N} \sumstar_{a\modd{D}} Y(a,D;f_1,f_2)\tZ(a, D)
+N^2\sqrt{x}(\log x)^3\]
for some pair of functions $f_1,f_2=f$or $g$.  We consider the case in
which $f_1=f_2=f$, the others being similar, or easier.  Since
\[\sum_{D\leq 2N} \sumstar_{a\modd{D}} Y(a,D;f,f)\tZ(a, D)=
\sum_{D\leq 2N} \sumstar_{b\modd{D}} Y_f(b,D)\tZ(b^2, D)\]
it follows from Proposition \ref{prop:EMN} that
\begin{align}\label{eqn:T'b}
T'(\U_1, \U_2, J_1, J_2)&\ll
\sum_{D\leq 2N} Y_{f,f}(D)\sumstar_{b\modd{D}}\tZ(b^2,D)+\frac{xN}{\log^C x}.
\end{align}
This holds for any $C>0$, and since there are
$O((\log x)^{6L})$ possible regions $\U_1\times\U_2\times
J_1\times J_2$, we see that the final term contributes $O(xN(\log
x)^{6L-C})$ in Proposition \ref{prop:bilinear3b}.  This is
satisfactory on taking $C\ge A+6L$.

To handle the main terms we note 
\begin{align}\label{eqn:YJ}
Y(J_i,f) = |J_i| + O(\sqrt{x}\exp(-\sqrt{\log x}))
\end{align}
by the Prime Number
Theorem. It follows that $Y(J_i,f)\ll |J_i|\ll \omega X$, 
and hence that 
$Y_{f,f}(D)\ll\omega^2X^2/\phi(D)$.  The main
terms in \eqref{eqn:T'b} are thus
\[\ll\omega^2X^2\sum_{D\leq 2N} \frac{\tau(D)}{\phi(D)}\sumstar_{a\modd{D}}
\tZ(a, D),\]
since any residue class $a$ coprime to $D$ arises $O(\tau(D))$ times
as a square.

Hence, to establish Proposition \ref{prop:bilinear3b} it will suffice
to show that
\begin{align}\label{eqn:E1}
\E &:= \omega^2X^2\sum_{\substack{\U_1,\U_2,J_1,J_2\\ 
\Cf_2(\U_1, \U_2, J_1,J_2) \textup{ or }\Cf_3(\U_1, \U_2, J_1,J_2)}}\sum_{D\leq 2N} \frac{\tau(D)}{\phi(D)} 
\sumstar_{a\modd{D}} \tZ(a, D) \\
&\ll \frac{xN}{\log^{A} x}
\end{align}  
when $L\ge A+6$.
For this we use the following two lemmas.

\begin{lem}\label{lem:UDJcount}
For fixed $\U_1, \U_2$, $D$ and $J_1$, the number of choices for $J_2$
subject to the condition $\Cf_2(\U_1, \U_2, J_1, J_2)$ is $\ll 1$. 
\end{lem}
\begin{proof}
Let $X_i = \inf J_i$, and fix $Z_i \in \U_i$. Then $\Cf_2(\U_1, \U_2,
J_1, J_2)$ implies that there exists $q_i \in J_i$ and $z_i \in \U_i$
such that 
\begin{equation}
D\sqrt{\frac x N} \geq |q_1z_2 - q_2z_1| = |X_1Z_2 - X_2Z_1| +
O(\omega X\sqrt{N}), 
\end{equation}and that there exists $q_i' \in J_i$ and $z_i' \in \U_i$ such that
\begin{equation}
D\sqrt{\frac x N} < |q_1'z_2' - q_2'z_1'| = |X_1Z_2 - X_2Z_1| +
O(\omega X\sqrt{N}). 
\end{equation}
Then 
\begin{equation}
|X_1Z_2 - X_2Z_1| = D\sqrt{\frac x N} +  O(\omega X\sqrt{N}),
\end{equation}and since $|Z_1| \gg \sqrt{N}$ we have
\begin{equation}
|X_2 - \frac{X_1Z_2}{Z_1}| = \frac{D}{|Z_1|}\sqrt{\frac x N} +  
O(\omega X).
\end{equation}
Since the different values for $X_2$ increase in steps of length 
$\asymp \omega X$ it follows that $\U_1, \U_2, D$ and $J_1$ determine
$O(1)$ choices for $J_2$.
\end{proof} 

\begin{lem}\label{lem:UDJcount2}
For fixed $\U_1, \U_2$, $D$ and $J_1$, the number of choices for $J_2$
subject to the condition $\Cf_3(\U_1, \U_2, J_1, J_2)$ is $\ll
\omega^{-1}(\log x)^{-A-2}.$ 
\end{lem}
\begin{proof}
Let $X_i = \inf J_i$, and fix $Z_i \in \U_i$.  For all $(z_1, z_2) \in \U_1\times \U_2$, we have that $\Delta(z_1, z_2) = \Delta(Z_1, Z_2) + O(\omega N)$.  Then
\begin{equation}
|X_1Z_2 - X_2 Z_1| \ll (\Df_0 + \omega N) \sqrt{\frac{x}{N}},
\end{equation}and since $|Z_1| \gg \sqrt{N}$ and $\omega N \ll 
\Df_0$ for $L\ge A+6$ this yields
\begin{equation}
X_2 = X_1 \frac{Z_2}{Z_1} + O\left(\frac{\Df_0}{N} \sqrt{x}\right).
\end{equation}This tells us that $X_2$ is restricted to be in an
interval of length $\ll \sqrt{x}(\log x)^{-A-6}$.  Since $J_2$ is
of length $\omega_1 X \gg \omega \sqrt{x}(\log x)^{-4}$, this gives
$O(\omega^{-1}(\log x)^{-A-2})$ choices for $J_2$. 
\end{proof}

By Lemma \ref{lem:UDJcount}, Lemma \ref{lem:UDJcount2} and \eqref{eqn:E1},
\begin{eqnarray*}
\E &\ll & \left(\omega + (\log x)^{-A-2}\right)
X^2\sum_{\U_1,\U_2}\sum_{D\leq 2N} \frac{\tau(D)}{\phi(D)}  
\sumstar_{a\modd{D}} \tZ(a, D) \\
&\ll&  \left(\omega + (\log x)^{-A-2}\right) X^2 N\sum_{D\leq 2N}
\frac{\tau(D)}{\phi(D)},  
\end{eqnarray*}where we have used Lemma \ref{lem:ZbD} for the last
line.  Hence if $L\ge A+6$ we obtain
\begin{eqnarray*}
\E \ll \left(\omega +(\log x)^{-A-2}\right)  X^2N(\log x)^2 \ll
\frac{xN}{\log^{L-2} x} + \frac{xN}{\log^A x}\ll \frac{xN}{\log^A x}, 
\end{eqnarray*}
as required.

\section{Proof of Proposition \ref{prop:bilinear3a}; 
Further  Manoeuvres} 
\label{section:bilinearpf3a}

Supposing that one of the functions $h_1$ and $h_2$ is $g$, we have according to Proposition \ref{prop:EMN} that
\begin{equation}
\sum_{D\leq 2N} \sumstar_{a\modd{D}} Y(a,D;h_1,h_2)Z(a, D) = 
\sum_{D\leq 2N} \sumstar_{a\modd{D}} Y_{h_1,h_2}(D)Z(a, D) + O_C\bfrac{xN}{\log^C x},
\end{equation}
for any $C>0$.  In the remaining case $h_1=h_2 = f$, we have
\begin{equation}
\sum_{D\leq 2N} \sumstar_{a\modd{D}} Y(a,D;f,f)Z(a, D) = 
\sum_{D\leq 2N} \sumstar_{b\modd{D}} Y_{f,f}(D)Z(b^2,D)+ O_C\bfrac{xN}{\log^C x}.
\end{equation}
From \eqref{eqn:YJ} we see that we may replace $Y_{h_1,h_2}(D)$ by
$|J_1||J_2|/\phi(D)$ in each case, with a total error
\begin{eqnarray*}
&\ll & x\exp(-\sqrt{\log x}) \sum_{\U_1,\U_2}\sum_{D\leq 2N}
\frac{1}{\phi(D)}\left(\sumstar_{a\modd{D}} |Z(a, D)| 
+\sumstar_{b\modd{D}} |Z(b^2, D)| \right)\\
&\ll& x\exp(-\sqrt{\log x})\sum_{D\leq 2N}
\frac{\tau(D)}{\phi(D)}\sum_{\U_1,\U_2}\sumstar_{a\modd{D}} \tZ(a, D) \\
&\ll& x\exp(-\sqrt{\log x}) \sum_{D\leq 2N} \frac{\tau(D)}{\phi(D)}N \\
&\ll & x\exp(-\sqrt{\log x})N\log x,
\end{eqnarray*}
by Lemma \ref{lem:ZbD}.  This is
satisfactory for the proposition.

It therefore remains to show that
\[|J_1||J_2|\sum_{\substack{\U_1, \U_2\\ \Cf_1(\U_1, \U_2, J_1, J_2)}} \sum_{D\leq 2N}\frac{1}{\phi(D)}\left(
\sumstar_{b\modd{D}}Z(b^2, D)-\sumstar_{a\modd{D}}Z(a, D)\right)
\ll\frac{xN}{\log^{A+2L} x},\]
or indeed that
\[\E:=\sum_{\substack{\U_1, \U_2\\ \Cf_1(\U_1, \U_2, J_1, J_2)}} \sum_{D}\frac{1}{\phi(D)}\left(
\sumstar_{b\modd{D}}Z(b^2, D)-\sumstar_{a\modd{D}}Z(a, D)\right)
\ll\frac{N}{\log^A x}.\]
Here we have dropped the condition $D\le 2N$, which follows
automatically since $\beta_z$ is supported on $N(z)\le 2N$.

It turns out that it is easier to estimate the corresponding sum in
which the factor $1/\phi(D)$ is replaced by $D/\phi(D)$.  To make this
transition we will use summation by parts on $D$ to see that 
\begin{equation}
\E \ll \sum_{\U_1, \U_2}\frac{|\E_0(t_0)|}{\Df_0},
\end{equation}where $t_0 = t_0(\U_1, U_2) > \Df_0$ is such that
$|\E_0(t)|$ is maximal, with 
\begin{equation}
\E_0(t)=\E_0(t,\U_1,\U_2) := \sum_{D < t} \frac{D}{\phi(D)}\left(
\sumstar_{b\modd{D}}Z(b^2, D)-\sumstar_{a\modd{D}}Z(a, D)\right).
\end{equation}

We now wish to remove the condition $D<t$ by arguing that the
contribution of those $\U_1$, $\U_2$ such that there exists $z_1, z_1'
\in \U_1$, $z_2, z_2' \in \U_2$ with $\Delta(z_1, z_2) < t$ and
$\Delta(z_1', z_2') \geq t$ is negligible.  Indeed, for such $\U_1,
\U_2$, we must have 
\begin{equation}
\Delta(z_1, z_2) = t + O(N\omega) = t\left(1+O((\log x)^{A+6-L})\right)
\end{equation} 
for all $z_1 \in \U_1$, $z_2 \in \U_2$.  If $\U_1=\U(c_1, \theta_1)$ and
$\U_2=\U(c_2, \theta_2)$, then  
$$\sin(\theta_1 - \theta_2) = \frac{t}{c_1c_2N}
\left(1+ O((\log x)^{A+6-L})\right).
$$
In general that if one restricts $\sin(\phi)$ to an interval of length
$\mu$, then $\phi$ will be confined to a set of measure
$O(\sqrt{\mu})$ modulo $\pi/2$ and hence modulo $2 \pi$ as well.  Indeed, suppose $\phi_1, \phi_2 \in [0, \pi/2 - \sqrt \mu]$, and that $|\sin \theta_1 - \sin \theta_2| \ll \mu$.  We want to show that $|\theta_1 - \theta_2| \ll \sqrt \mu$.  This follows from the elementary formula  
$$\sin \phi_1 - \sin \phi_2 = 
2 \cos \bfrac{\phi_1+\phi_2}{2} \sin \bfrac{\phi_1 - \phi_2}{2},$$  
provided that $\cos \bfrac{\phi_1+\phi_2}{2} \gg \sqrt{\mu}$.  The latter bound follows from $\pi/2  - (\phi_1 + \phi_2)/2  \gg \sqrt \mu$.

Hence if we fix
$c_1, \theta_1$ and $c_2$  then the total number of
choices for $\theta_2$ is $\ll (\log x)^{(A+6+L)/2}$.  This
gives $\ll(\log x)^{(A+6+7L)/2}$ total possibilities for $(\U_1,
\U_2)$.  A given residue class $a$ occurs $O(\tau(D))$ times as $b^2$,
and $D/\phi(D)\ll\log x$.  Thus
each pair $(\U_1,\U_2)$ contributes $\ll \omega^4 N^2(\log x)^{17}$ to the
sum, by Lemma \ref{lem:ZbD},
giving a total contribution of  
\begin{equation}
\ll \frac{N^2(\log x)^{(A+40-L)/2}}{\Df_0} \ll
\frac{N}{(\log x)^{(L-3A-52)/2}} 
\end{equation} to $\E$, which suffices when $L =6A+52$. 

It now suffices to show that
\begin{equation}\label{eqn:E1bdd2}
\E_1= \E_1(\U_1, \U_2) = \sum_{D}\frac{D}{\phi(D)}\left(
\sumstar_{b\modd{D}}Z(b^2, D)-\sumstar_{a\modd{D}}Z(a, D)\right) \ll
\frac{N^2}{\log^{C_1}x}, 
\end{equation}for any $C_1>0$ and for fixed $\U_1, \U_2$.
Since we have $\Delta>0$ whenever
$(z_1,z_2)\in\U_1\times\U_2$ we see that if $D\mid\Delta$ then
\[\sum_{k\mid\Delta/D}\mu(k)=\left\{\begin{array}{cc}
1, & D=\Delta,\\  0, & D\not=\Delta.\end{array}\right.\] 
We therefore have
\[\E_1=\sum_{D=1}^{\infty}\sum_{k=1}^\infty 
\frac{D\mu(k)}{\phi(D)}\left(
\sumstar_{b\modd{D}}W(b^2,k,D)-\sumstar_{a\modd{D}}W(a,k,D)\right)\]
where 
\[W(a,k,D):=\sumb_{\substack{(z_1, z_2) \in \U_1 \times \U_2\\
kD\mid\Delta\\ az_2 \equiv z_1 \modd{D}}}\beta_{z_1}\beta_{z_2}.\]
When $kD\mid\Delta$, there is a unique integer
$c=c(z_1,z_2;kD)$ modulo $kD$ such that $cz_2 \equiv z_1 \modd{kD}$, and
conversely this congruence implies that $kD\mid\Delta$.  For this
integer $c$ we have $(c,kD)=1$ 
and
\[\#\{b\modd{D}: b^2z_2 \equiv z_1 \modd{D}\}=
\#\{b\modd{D}: b^2\equiv c\modd{D}\}=
\sum_{\substack{\chi\modd{D}\\ \chi^2=\chi_0}}\chi(c).\]

It now follows that
\[\sumstar_{b\modd{D}}W(b^2,k,D)-\sumstar_{a\modd{D}}W(a,k,D)
=\sum_{\substack{\chi\modd{D}\\ \chi^2=\chi_0\\ \chi\not=\chi_0}}\sumstar_{c\modd{kD}}
\sumb_{\substack{(z_1,z_2)\in\U_1\times\U_2\\ cz_2\equiv z_1\modd{kD}}}
\beta_{z_1}\beta_{z_2}\chi(c),\]
and hence that
\[\E_1=\sum_{D=1}^{\infty}\sum_{k=1}^\infty 
\frac{D\mu(k)}{\phi(D)}
\sum_{\substack{\chi\modd{D}\\ \chi^2=\chi_0\\ \chi\not=\chi_0}}\sumstar_{c\modd{kD}}
\sumb_{\substack{(z_1,z_2)\in\U_1\times\U_2\\ cz_2\equiv z_1\modd{kD}}}
\beta_{z_1}\beta_{z_2}\chi(c),\]
Let $d=d(\chi)$ be the conductor of $\chi$ and write $D=de$ and
$ek=f$, giving
\[\E_1=\sum_{d>1}\sum_{f}C(d,f)\sumstar_{\substack{\chi\modd{d}\\ \chi^2=\chi_0}}
\sumstar_{c\modd{df}}
\sumb_{\substack{(z_1,z_2)\in\U_1\times\U_2\\ cz_2\equiv z_1\modd{df}}}
\beta_{z_1}\beta_{z_2}\chi(c),\]
where
\[C(d,f)=\sum_{ek=f}\frac{de\mu(k)}{\phi(de)}=
\frac{d}{\phi(d)}\sum_{ek=f}\frac{\phi(d)e\mu(k)}{\phi(de)}.\]
Note that the sum for $\chi\modd{d}$ is empty unless $d=d_1$ or $4d_1$
or $8d_1$, with $d_1$ odd and square-free, in which cases there are at
most two possible characters $\chi$.
When $d$ is given, the function $\kappa(e):=\phi(d)e/\phi(de)$ is
multiplicative in $e$.  Moreover if $v\ge 1$ then
\[(\kappa*\mu)(p^v)=\left\{\begin{array}{cc} (p-1)^{-1}, & \mbox{if $v=1$
  and $p\nmid d$,}\\ 0, & \mbox{otherwise.}\end{array}\right.\]
We then see that
\[C(d,f)=\frac{d\mu^2(f)}{\phi(df)}\]
if $(d,f)=1$ and $C(d,f)=0$ otherwise.  This leads to the expression
\begin{align}\label{eqn:E1b}
\E_1=\sum_{\substack{f,d\\ (d,f)=1}}\frac{d\mu^2(f)}{\phi(df)}
\sumstar_{\substack{\chi\modd{d}\\ \chi^2=\chi_0\\ \chi\not=\chi_0}}
\left(\sumstar_{c\modd{df}}\sumb_{\substack{(z_1,z_2)\in\U_1\times\U_2\\ cz_2\equiv
      z_1\modd{df}}} \beta_{z_1}\beta_{z_2}\chi(c)\right).
\end{align}

We proceed to show that large values of $f$ make a negligible
contribution.  Since $df\mid\Delta(z_1,z_2)$ we will have $df\le 2N$.
On recalling that
$0\le\beta_z\le 1$ we find that
\begin{eqnarray*}
\lefteqn{\sum_{f>F}\sum_{\substack{d\\ (d,f)=1}}\frac{d\mu^2(f)}{\phi(df)}
\sumstar_{\substack{\chi\modd{d}\\ \chi^2=\chi_0}}
\left|\sumstar_{c\modd{df}}\sumb_{\substack{(z_1,z_2)\in\U_1\times\U_2\\ cz_2\equiv
      z_1\modd{df}}} 
\beta_{z_1}\beta_{z_2}\chi(c)\right|}\hspace{3cm}\\
&\ll& (\log x)\sum_{f>F}f^{-1}\sum_{d\le 2N/f} \sum_{\substack{df|D\\D \le
    2N}}\sumstar_{a \modd{D}} \tZ(a, D)\\ 
&\ll&(\log x)\sum_{f>F}f^{-1}\sum_{d\le 2N/f} \sum_{\substack{df|D\\D \le 2N}} N\\
&\ll&\frac{N^2(\log x)^2}{F},
\end{eqnarray*}
by Lemma \ref{lem:ZbD}.  An inspection of \eqref{eqn:E1bdd2} and
\eqref{eqn:E1b} reveals that this 
is satisfactory if we take
\[F=(\log x)^{C_1+2}.\]

Next, we split our sum into three different ranges for $d$, namely $d
\leq D_1$, $D_1 < d \le D_2$, $d> D_2$ where 
\[D_1 = F^{10} (\log x)^{2C_1+14}\;\;\;\mbox{and}\;\;\;
D_2 = \frac{N}{F^{15}(\log x)^{3C_1+21}}.  \]
We deal with the middle range
for $d$ in Section \ref{sec:middled}, and the remaining ranges
in Section \ref{sec:smalllarged}. 

\section{Proof of Proposition \ref{prop:bilinear3a}; 
Middle $d$} \label{sec:middled}

We proceed to dispose of the middle range of values for $d$, making use
of a major intermediate result from the work of Friedlander and
Iwaniec \cite{FI}. Thus we will consider the case in which $D<d\le 2D$
say.  (We should make it clear to the reader that there is no longer any
connection between $D$ and $\Delta$.) We will set
\[\E_1(D):=\sum_{f\le F}f^{-1}\mu^2(f)\sum_{\substack{D<d\le 2D\\ (d,f)=1}}
\sumstar_{\substack{\chi\modd{d}\\ \chi^2=\chi_0}}\left|\sumstar_{c\modd{df}}
\sumb_{\substack{(z_1,z_2)\in\U_1\times\U_2\\ cz_2\equiv z_1\modd{df}}} 
\beta_{z_1}\beta_{z_2}\chi(c)\right|.\]

Let $d=d_1d_2$ where $d_1$ is odd and square-free, and $d_2$ is 1, 4,
or 8, and write $\chi=\chi_1\chi_2$ accordingly, so that $\chi_1(n)$
is the Jacobi symbol $(\frac{n}{d_1})$.  If we set $g=fd_2$ then $g$
and $d_1$ will be coprime.  We now split
the variables $z_j$ into congruence classes $z_j\equiv w_j\modd{g}$ 
and find that
\[\E_1(D)\ll \sum_{f\le F}\; \sum_{d_2=1,4,8}\;
\sum_{w_1,w_2\modd{g}}\E_2(D;g,w_1,w_2),\]
where
\[\E_2(D;g,w_1,w_2):=\sum_{\substack{D/8<d_1\le 2D\\ (d_1,g)=1}}
\left|\sumstar_{b\modd{d_1}}
\sumb_{\substack{(z_1,z_2)\in\U_1\times\U_2\\ z_j\equiv w_j\modd{g}
\\ bz_2\equiv z_1\modd{d_1}}} 
\beta_{z_1}\beta_{z_2}\left(\frac{b}{d_1}\right)\right|.\]
We now write $\gamma_{z_j}=\beta_{z_j}$ if $z_j\equiv w_j\modd{g}$, and
$\gamma_{z_j}=0$ otherwise, so that $0\le\gamma_1,\gamma_2\le 1$. It
follows that there is some choice of $f,d_2$ and $w_1,w_2$ such that
\begin{align}\label{eqn:E1c}
\E_1(D)&\ll F^5\E_3(D)
\end{align}
with 
\[\E_3(D)=\sum_{D/8<d\le 2D}\mu^2(2d)\left|\sumstar_{b\modd{d}}\;
\sumb_{\substack{(z_1,z_2)\in\U_1\times\U_2\\ 
 \\ bz_2\equiv z_1\modd{d}}} 
\gamma_{z_1}\gamma_{z_2}\left(\frac{b}{d}\right)\right|.\]
(Here we have replaced the dummy variable $d_1$ by $d$, for notational
convenience.) 

Writing $z_j=x_j+iy_j$ for $j=1,2$ we classify the numbers $z_j$
according to the value of $(x_1,d)=r$, say. Since $bz_2\equiv z_1\modd{d}$ 
with $(b,d)=1$ we will have $(x_2,d)=r$ also. We now write $d=rs$
and $x_j=ru_j$ for $j=1,2$, so that $(u_j,s)=1$. Since $z_j$ is
primitive and $r\mid x_j$ it follows that $(y_j,r)=1$.  Then $d\mid
x_1y_2-x_2y_1$ if and only if there is a $b$
such that $bz_2\equiv z_1\modd{d}$.  Moreover $d\mid
x_1y_2-x_2y_1$ if and only if
$u_1y_2-u_2y_1$. Thus there is an integer $c$ such that
$y_1u_1^{-1}\equiv y_2u_2^{-1}\equiv c\modd{s}$. 
We will also have $bu_2\equiv u_1\modd{s}$ and
$by_2\equiv y_1\modd{r}$, so that
\[\left(\frac{b}{d}\right)=\left(\frac{b}{r}\right)\left(\frac{b}{s}\right)
=\left(\frac{y_1y_2}{r}\right)\left(\frac{u_1u_2}{s}\right).\]
These considerations show that
\[\E_1(D)\ll 
F^3\sum_{D/8<rs\le 2D}\sum_{c\modd{s}}\mu^2(2rs)|\Sigma_1(r,s,c)\Sigma_2(r,s,c)|,\]
with
\[\Sigma_j(r,s,c):=
\sumb_{\substack{ru_j+iy_j\in\U_j \\ y_ju_j^{-1}\equiv c\modd{s}}} 
\gamma_{ru_j+iy_j}\left(\frac{y_j}{r}\right)\left(\frac{u_j}{s}\right).\]
We have dropped the condition $(u_j,s)=1$ since the Legendre symbol
vanishes when this fails to hold.

By Cauchy's inequality we deduce that for either $j=1$ or $j=2$,
\begin{equation}\label{eqn:E1c2}
\E_1(D)\ll F^5\sum_{r\leq 2D}\;\sum_{D/8r<s\le 2D/r}\mu^2(2rs)
\sum_{c\modd{s}}|\Sigma_j(r,s,c)|^2.
\end{equation}
Suppressing the dependence on $r$ and $j$
we may write
\[\Sigma_j(r,s,c):=\sum_{yu^{-1}\equiv c\modd{s}}
\alpha_{u,y}\left(\frac{u}{s}\right)\]
where
\[\alpha_{u,y}=\gamma_{ru+iy}\left(\frac{y}{r}\right)\]
when $z=ru+iy\in\U_j$ and the conditions for $\sumb$ hold, and
$\alpha_{u,y}=0$ otherwise. The coefficients $\alpha_{u,y}$ are
therefore supported in $|u|\leq \sqrt{2N}/r$ and
$|y|\leq\sqrt{2N}$. Since $\beta_z$ is supported on values with $z$
primitive and $\tReb(z)$ odd, we may assume that $(u,2y)=1$.

We are now ready to apply Proposition 14.1 of Friedlander and Iwaniec
\cite{FI}. This shows that if $\alpha_{u,y}$ is supported in $U<u<2U$
and $Y<y<2Y$, then
\[\sum_{S<s\le 2S}\mu^2(2s)\sum_{c\modd{s}}\left|\sum_{yu^{-1}\equiv c\modd{s}}
\alpha_{u,y}\left(\frac{u}{s}\right)\right|^2\le 
N(S,U,Y)\sum_{u,y}\tau(u)|\alpha_{u,y}|^2\]
with
\[N(S,U,Y)\ll_{\epsilon} S+S^{-1/2}UY+S^{1/3}(UY)^{2/3}(\log UY)^4+
(U+Y)^{1/12}(UY)^{11/12+\epsilon}\]
for any fixed $\epsilon>0$.

In our case we have
\[\sum_{u,y}\tau(u)|\alpha_{u,y}|^2\ll UY(\log x).\]
Thus, if we sum over dyadic ranges for $u$, $y$, and $s$, we obtain
\begin{eqnarray*}
\lefteqn{\sum_{D/8r<s\le 2D/r}\mu^2(2rs)\sum_{c\modd{s}}|\Sigma_j(r,s,c)|^2}\\
&\ll& \left\{\frac{D}{r}+\left(\frac{D}{r}\right)^{-1/2}\frac{N}{r}
+\left(\frac{D}{r}\right)^{1/3}\left(\frac{N}{r}\right)^{2/3}(\log x)^4
+N^{1/24}\left(\frac{N}{r}\right)^{11/12}N^{\epsilon}
\right\}\frac{N}{r}(\log x)\\ 
&\ll&
\{D+D^{-1/2}N+D^{1/3}N^{2/3}+N^{23/24+\epsilon}\}\frac{N}{r}(\log
x)^5.
\end{eqnarray*}
Summing over $r$, we see that
from \eqref{eqn:E1c2} that
\[\E_1(D)\ll F^5 (\log x)^6
\{D+D^{-1/2}N+D^{1/3}N^{2/3}+N^{23/24+\epsilon}\}N.\]
Thus, on summing over dyadic ranges for $D$, we see that values of $d$
with $D_1\leq d\leq D_2$ make a satisfactory contribution given our
choices of $D_1$ and $D_2$.

\section{Proof of Proposition \ref{prop:bilinear3a}
: large $d$ and small $d$}\label{sec:smalllarged} 
\subsection{Large $d$}\label{sec:larged}
Now, we will show for fixed $f \leq F$, and any $C >0$ that
\begin{equation}
\sum_{\substack{d>D_2\\ (d,f)=1}}\frac{d}{\phi(d)}
\sumstar_{\substack{\chi\modd{d}\\ \chi^2=\chi_0}}
\left(\sumstar_{c\modd{df}}
\sumb_{\substack{(z_1,z_2)\in\U_1\times\U_2\\ cz_2\equiv z_1\modd{df}}} 
\beta_{z_1}\beta_{z_2}\chi(c)\right)\ll_C \frac{N^{2}}{\log^C x}.
\end{equation}  
In the sequel, we shall use
the convention that $C$ denotes a large positive constant, not
necessarily the same from line to line. 

As in the previous section we decompose $d$ as $d_1d_2$ and $\chi$ as
$\chi_1\chi_2$.  Writing $\Delta=\Delta(z_1,z_2)$ for short, we 
must have $df\mid\Delta$, and so we may set
$\Delta=d_1et$ where $e$ is odd and $t$ is a power of 2. Our
conditions on $\U_1$ and $\U_2$ ensure that $0<\Delta\le 2N$,
whence $1\le et\le 16N/D_2\ll (\log x)^{18C_1+51}$. We split the
sums over $z_j$ into congruence classes $z_j\equiv w_j\modd{8et}$, and
fix the parameters 
\begin{equation}\label{eqn:set}
f,\;\; d_2,\;\; \chi_2,\;\; e,\;\; w_1,\;\; w_2\;\; \mbox{and}\;\; t. 
\end{equation}
Each admissible pair $z_1,z_2$ corresponds to a unique integer
$k\modd{\Delta}$ with the property that $kz_2\equiv
z_1\modd{\Delta}$, and then
\[\chi(c)=\chi(k)=\chi_2(k)\left(\frac{k}{d_1}\right),\]
where $\chi_2(k)$ is determined by the parameters \eqref{eqn:set}.  The
number of choices for the parameters \eqref{eqn:set} is bounded by a fixed
power of $\log x$ so that it suffices to show that
\[\sum_{\substack{d_1>D_2/d_2\\ (d_1,2f)=1}}\frac{d_1\mu^2(d_1)}{\phi(d_1)}
\left(\sumstar_{k\modd{\Delta}}
\sumb_{z_1,z_2}\beta_{z_1}\beta_{z_2}\left(\frac{k}{d_1}\right)\right)\ll_C 
\frac{N^{2}}{\log^C x}\]
for every $C>0$, where the sum over $z_1,z_2$ satisfies the conditions
\[(z_1,z_2)\in\U_1\times\U_2,\;\; kz_2\equiv z_1\modd{\Delta}\;\;
z_j\equiv w_j\modd{8et}\;(j=1,2)\;\;\mbox{and}\;\;\Delta=d_1 et.\]

We proceed to investigate the Jacobi symbol which occurs here.  Our
analysis is very close to that given by Friedlander and Iwaniec 
\cite[Lemma 17.1]{FI}. However our situation is not quite the same as
theirs. In what follows we will make repeated use of the fact that
$z_j=x_j+iy_j$ with $x_j>0$ and $(x_j,2y_j)=1$.
If we set $r=(x_1,\Delta)$ then, as in the previous
section, we have $r=(x_2,\Delta)$, allowing us to write $x_j=ru_j$ for
$j=1$ and $2$.  Since the $x_j$ are odd we see that $r\mid d_1 e$, and
setting $s=d_1e/r$ we find that $ku_2\equiv u_1\modd{s}$
and $ky_2\equiv y_1\modd{r}$.  Then $(k,e)=1$ since $(k,\Delta)=1$,
and $(y_j,r)=1$, since $z_j$ is primitive.  We therefore find that
\[\left(\frac{k}{d_1}\right)=\left(\frac{k}{e}\right)
\left(\frac{k}{r}\right)\left(\frac{k}{s}\right)=\left(\frac{k}{e}\right)
\left(\frac{y_1y_2}{r}\right)\left(\frac{u_1u_2}{s}\right).\]
Recalling that $z_j \equiv w_j \modd{e}$, 
we see that $\left(\frac{k}{e}\right)$ is determined by $w_1$ and $w_2$.  By
quadratic reciprocity we have
\[\left(\frac{u_1u_2}{s}\right)=\left(\frac{s}{u_1u_2}\right)<s,u_1u_2>,\]
where
\[<a,b>:=(-1)^{(a-1)(b-1)/4}\]
for odd integers $a,b$. However $\Delta=x_1y_2-x_2y_1$, whence
$st=u_1y_2-u_2y_1$.  It follows that
\[\left(\frac{s}{u_1}\right)=\left(\frac{-tu_2y_1}{u_1}\right)\]
and
\[\left(\frac{s}{u_2}\right)=\left(\frac{tu_1y_2}{u_2}\right),\]
so that
\begin{eqnarray*}
\left(\frac{s}{u_1u_2}\right)&=&\left(\frac{t}{u_1u_2}\right)
\left(\frac{-1}{u_1}\right)\left(\frac{y_1}{u_1}\right)
\left(\frac{y_2}{u_2}\right)<u_1,u_2>\\
&=&\left(\frac{t}{x_1x_2}\right)
\left(\frac{-1}{x_1}\right)\left(\frac{-1}{r}\right)
\left(\frac{y_1}{u_1}\right)\left(\frac{y_2}{u_2}\right)<u_1,u_2>.
\end{eqnarray*}
The first two factors on the right are determined by the parameters
\eqref{eqn:set}, since $z_j\equiv w_j\modd{8}$, whence
\[\left(\frac{k}{d_1}\right)=\eta\left(\frac{y_1y_2}{r}\right)
<s,u_1u_2>\left(\frac{-1}{r}\right)
\left(\frac{y_1}{u_1}\right)\left(\frac{y_2}{u_2}\right)<u_1,u_2>,\]
for some value of $\eta=\pm 1$ determined by the parameters \eqref{eqn:set}.
We then deduce that
\[\left(\frac{k}{d_1}\right)=\eta\left(\frac{y_1}{x_1}\right)
\left(\frac{y_2}{x_2}\right)<r,r><s,u_1u_2><u_1,u_2>.\]
We can now use the relations $<ab,c>=<a,c><b,c>$ and
$<a,bc>=<a,b><a,c>$ to conclude that
\begin{align*}
\left(\frac{k}{d_1}\right)&=\eta\left(\frac{y_1}{x_1}\right)
\left(\frac{y_2}{x_2}\right)<sr, u_1u_2><x_1,x_2> \\
&= \eta\left(\frac{y_1}{x_1}\right)
\left(\frac{y_2}{x_2}\right)<d_1e, x_1x_2><x_1,x_2>.
\end{align*}
The parameters \eqref{eqn:set} determine $x_1, x_2 \modd{8et}$, and it follows that the factor $<x_1,x_2>$ is uniquely determined by the parameters \eqref{eqn:set}.  Moreover, 
$$d_1et = \Delta = \tIm (\bar z_1 z_2)  \equiv \tIm (\bar w_1w_2) \modd{8et}.
$$Writing $\tIm (\bar w_1 w_2) = etn$, we see that $d_1 \equiv n \modd{8}$ so that $d_1 \modd{8}$ is determined by the parameters \eqref{eqn:set}, and so $<d_1e, x_1x_2>$ is also.

We therefore see that it will be enough to show that
\[\sumb_{z_1,z_2}\frac{d_1\mu^2(d_1)}{\phi(d_1)}\beta_{z_1}
\left(\frac{y_1}{x_1}\right)\beta_{z_2}\left(\frac{y_2}{x_2}\right)
\ll_C \frac{N^{2}}{\log^C x}\]
for each given set of parameters \eqref{eqn:set}, where the sum over
$z_1,z_2$ is restricted by the conditions
\[(z_1,z_2)\in\U_1\times\U_2,\;\; z_j\equiv w_j\modd{8et}\;(j=1,2),
\;\; et\mid\Delta, \;\; d_1>D_2/d_2,\;\;\mbox{and}\;\;(d_1,2f)=1.\]
Moreover $d_1$ is then defined by $d_1:=\Delta/et$. We may remove the 
conditions for $\sumb$, since $\beta_z$ is supported on primitive
$z\equiv 1\modd{2}$ with positive real part. 
Moreover the condition $et\mid\Delta$ depends only on the
choice of $w_1$ and $w_2$.  Following the notation of Friedlander and Iwaniec
\cite[Section 17]{FI} we write
\[\beta_z'=\beta_z[z]=\beta_z i^{(x-1)/2}\left(\frac{y}{x}\right)\]
for $z=x+iy$. Since the factors corresponding to $i^{(x-1)/2}$ are
determined by $w_1$ and $w_2$ we then see that it suffices to show
that
\[\E_4:=\sum_{z_1,z_2}\frac{d_1\mu^2(d_1)}{\phi(d_1)}\beta_{z_1}'
\beta_{z_2}'\ll_C \frac{N^{2}}{\log^C x},\]
where the sum is subject to
\[(z_1,z_2)\in\U_1\times\U_2,\;\; z_j\equiv w_j\modd{8et}\;(j=1,2),
\;\; d_1>D_2/d_2,\;\;\mbox{and}\;\;(d_1,2f)=1,\]
with $d_1:=\Delta/et$ as before.

The multiplicative function
\[\kappa(n):=\left\{\begin{array}{cc} \frac{n\mu^2(n)}{\phi(n)}, & 
(n,2f)=1,\\ 0, & (n,2f)>1,\end{array}\right.\]
may be written as $\kappa=\kappa_0*1$, where $\kappa_0(n)=\mu(n)$ for
$n\mid 2f$, and for $p\nmid 2f$,
\[\kappa_0(p)=\frac{1}{p-1},\;\;\;\kappa_0(p^2)=\frac{-p}{p-1},\;\;\;
\mbox{and}\;\;\;\kappa_0(p^r)=0,\;(r\ge 3).\]
It follows that we may write
\[\E_4=\sum_{n=1}^{\infty}\kappa_0(n)\sum_{z_1,z_2}\beta_{z_1}'\beta_{z_2}'\]
where the inner sum is subject to
\begin{equation}\label{c*}
(z_1,z_2)\in\U_1\times\U_2,\;\; z_j\equiv w_j\modd{8et}\;(j=1,2),
\;\; \Delta>etD_2/d_2,\;\;\mbox{and}\;\; etn\mid\Delta.
\end{equation}
Since
\[\sumb_{\substack{(z_1,z_2)\in\U_1\times\U_2\\ n\mid\Delta}}\leq
\sum_{\substack{D\le 2N\\ n\mid D}}\sumstar_{a \modd{D}} \tZ(a, D)\ll N^2/n\]
by \eqref{eqn:lemZaD1}, the contribution to $\E_4$ from integers
$n\ge (\log x)^B$ say is
\begin{eqnarray*}
&\ll&\sum_{n\ge (\log x)^B} \frac{N^2}{n}|\kappa_0(n)|\\
&\ll&\sum_{n_1\mid 2f}\sum_{n_2n_3^2\ge (\log x)^B/n_1} 
\frac{N^2\log(n_2n_3)}{n_2^2n_3^2}\\
&\ll&\sum_{n_1\mid 2f}\sum_{m^2\ge (\log x)^B/n_1} 
\frac{N^2\tau(m)\log m}{m^2}\\
&\ll&\sum_{n_1\mid 2f}N^2(\log x)^{1-B/2}n_1^{1/2}\\
&\ll & FN^2 (\log x)^{1-B/2}.
\end{eqnarray*}
In order to achieve a bound $O(N^2(\log x)^{-C})$ for $\E_4$ we choose
$B$ so that $F(\log x)^{1-B/2}=(\log x)^{-C}$.  Thus
it will now be enough to show that
\[\sum_{z_1,z_2}\beta_{z_1}'\beta_{z_2}'\ll_C \frac{N^{2}}{\log^{C+B} x}\]
for every fixed $C>0$, for each individual value of $n\le (\log
x)^B$.  The sum is again subject to \eqref{c*}, and we can allow for
the constraint $etn\mid\Delta$ by subdividing the sum according to the
values of $z_j$ modulo $8etn$.  Thus there are Gaussian integers
$w_1,w_2$ for which it will suffice to show that
\[\sum_{\substack{(z_1,z_2)\in\U_1\times\U_2\\ z_j\equiv w_j\modd{8etn}\\
\Delta>etD_2/d_2}}\beta_{z_1}'\beta_{z_2}'\ll_C \frac{N^{2}}{\log^C x}\]
for every fixed $C>0$, and for each choice of $e,t,n\le (\log x)^C$ 
and of $w_1$ and $w_2$.

Our next task is to remove the condition $\Delta>etD_2/d_2$.  This may
be done by further subdividing the regions $\U_i$ and following the
methods of Section \ref{section:bilinearpf3b}, used to prove
Proposition \ref{prop:bilinear3b}.  We do not repeat the details.

%

We now have to handle
\[\sum_{\substack{(z_1,z_2)\in\U_1\times\U_2\\ z_j\equiv
    w_j\modd{8etn}}}\beta_{z_1}'\beta_{z_2}'
=\left\{\sum_{\substack{z\in\U_1\\ z\equiv
      w_1\modd{8etn}}}\beta_z'\right\}
\left\{\sum_{\substack{z\in\U_2\\ z\equiv w_2\modd{8etn}}}\beta_z'\right\}\]
we note that $8etn$ is at most a power of $\log x$, while $\beta_z'$ is
supported on values for which every prime factor of $N(z)$ is at least
$x^{\delta}$.  Thus $z$ is automatically coprime to
$8etn$, and so we may assume that $w_1$ and $w_2$ are coprime to
$8etn$. This allows us to pick out the congruence conditions
$z\equiv w\modd{8etn}$ using multiplicative characters
$\chi\modd{8etn}$ over $\mathbb{Z}[i]$. We therefore deduce that
\[\sum_{\substack{z\in\U_j\\ z\equiv w\modd{8etn}}}\beta_z'=
\frac{1}{\phi_{\mathbb{Q}[i]}(8etn)}\sum_{\chi\modd{8etn}}
\overline{\chi}(w)S(\chi,\U_j),\]
where $\phi_{\mathbb{Q}(i)}(q)$ is the Euler $\phi$-function for the
Gaussian integers, and 
$$S(\chi, \U) = \sum_{z \in \U} \beta_z' \chi(z).$$
We have then reduced our problem to one of showing that for all $C>0$,
\begin{equation}
S(\chi, \U) \ll_C N (\log x)^{-C},
\end{equation}
for all $\chi\modd{8etn}$ and $\U = \U_1$ or $\U= \U_2$.  From now on
it will be convenient to set $f=8etn$, so that $\chi$ is a character
to modulus $f$. The reader should note that we are recycling some of
our previous notation --- the symbol $f$ no longer has its former meaning!

We recall here that $\U$ is of the form 
\begin{equation}\label{Uf}
\U = \{z: \sqrt{N'} < |z| \leq (1+\omega_1)\sqrt{N'}, \theta_0 < \arg
z < \theta_0 +\omega_2\}, 
\end{equation}
where $N' \asymp N$.  We pick out the condition $\arg z \in (\theta_0,
\theta_0+\omega_2)$ using a twice continuously differentiable
periodic function $w(\theta)$, where 
\begin{equation}
w(\theta)
\begin{cases}
= 1 &\textup{ if $\theta \in (\theta_0, \theta_0+\omega_2)\modd{2\pi}$}\\
= 0 &\textup{ if $\theta \notin [\theta_0-\log^{-C}x,
  \theta_0+\omega_2+\log^{-C}x)]\modd{2\pi}$}, 
\end{cases}
\end{equation}and where $|w''(\theta)| \ll \log^{2C} x$.  Then
\begin{equation}
S(\chi, \U)  =\sum_{N' < N(z) \le N'(1+\omega)} \beta_z' \chi(z) 
w(\arg z) + O\bfrac{N}{\log^C x}.
\end{equation}
The Fourier coefficients of $w$ satisfy $c_k\ll k^{-2}\log^{2C} x$ for
$k\not=0$, whence
\begin{equation*}
w(\arg z) = \sum_k c_k \bfrac{z}{|z|}^k = \sum_{|k| \leq \log^{3C}x} 
c_k \bfrac{z}{|z|}^k + O(\log^{-C} x).
\end{equation*}
It then suffices to show that
\begin{equation}
S(\chi, N', k) :=
\sum_{N' < N(z) \le N'(1+\omega)} \beta_z' \chi(z) \bfrac{z}{|z|}^k
\ll_C N (\log x)^{-4C}
\end{equation}
for any $C >0$, and for $|k|\le\log^{3C}x$.  Indeed we will do rather
better, and show that one can achieve a small power saving in $N$.

Recall that $\beta_z = \beta_{N(z)}$, where $\beta_n$ is the indicator
function of a set of one of the shapes
\begin{eqnarray*}
Q_j:&=&\{p_1...p_{j+1}\in (N', N'(1+\omega)]: p_{j+1} \in J,\;
 p_{j+1} < ... < p_1 < Y,\\
 && \hspace{2cm}p_1...p_j <Y\le p_1...p_{j+1}<x^{1/2-\delta} \} 
\end{eqnarray*}
or
\[ R:=\{n\in (N', N'(1+\omega)]: (n, P(V)) = 1\}.\] 
Here we will have $0\le j\le n_0=\left[\frac{\log Y}{\delta\log
    x}\right]$, and $J = [V, V(1+\kappa))\subseteq [x^{\delta},Y)$.  
In particular we interpret $Q_0$ to be $\{p: p\in
J\cap(N',N'(1+\omega)]\}$.

It will be convenient to write
\begin{equation}\label{ld}
\lambda(n) = \sump_{N(z) = n} \chi(z) \bfrac{z}{|z|}^k [z],
\end{equation}
where $\sump$ denotes a sum over primitive $z$. Then
\[S(\chi, N', k)=\sum_n \lambda(n),\]
where $n$ runs over $R$, or one of the sets $Q_j$.  We discuss the
procedure in the case of $Q_j$, the situation for the set $R$ being
similar, or indeed easier.  One small difference is that elements of
$R$ need not be square-free.  However integers $n\in R$ which are not
square-free make a negligible contribution, since any prime factor
must be at least $x^{\delta}$.

We begin by
handling the terms in which the largest prime factor of $n$, which we
write as $P(n)$ say, exceeds $N^{99/100}$. The contribution from such integers 
is
\[\sum_{m\le 2N^{1/100}}\;\sum_{\substack{p>\max(P(m),N^{99/100})\\ mp\in
    Q_j}}\lambda(mp).\]
Since $p$ is the largest prime factor of $mp$ one sees from the
definition of the set $Q_j$ that one may rewrite the conditions
$p>P(m)$ and $mp\in Q_j$ to say that $p$ runs over a certain interval
$I_j(m)\subseteq [N/m,2N/m)$.  We may then apply the following result
of Iwaniec and Friedlander \cite[Theorem $2^{\psi}$]{FI}

\begin{lem} \label{lem:JKprime}
For any $m \geq 1$ we have
\begin{equation}
\sum_{n\leq U} \Lambda(n) \lambda(mn) \ll f(|k|+1)m U^{76/77},
\end{equation}
where $\lambda$ is given by \eqref{ld} and $\chi$ is a character modulo $f$.
\end{lem}

Note that in our situation $fk$ is at most a power of $\log x$.
The above result shows that
\begin{eqnarray*}
\sum_{m\le 2N^{1/100}}\sum_{\substack{p>\max(P(m),N^{99/100})\\ mp\in
    Q_j}}\lambda(mp)&\ll & f(|k|+1)\sum_{m\le 2N^{1/100}}m(N/m)^{76/77}\\
&\ll& f(|k|+1) N^{76/77+(78/77)/100}.
\end{eqnarray*}
Since $76/77+(78/77)/100<1$ this is satisfactory.

We now examine the terms in which every prime factor of $n$ is at most
$N^{99/100}$.  To do this, we first rewrite our sum in terms of bilinear sums.  Let $n=p_1\ldots p_{j+1}$, as in the description of
the set $Q_j$, and divide the range for each prime $p_i$ into intervals of the form $(P_i,2P_i]$.  This will give us at most $(2\log N)^{1+n_0}$
sets of dyadic ranges, and since $n_0\ll\delta^{-1}= (\log
x)^{1-\varpi}$ there will be $O(N^{\epsilon})$ sets of
ranges. Moreover we may suppose that
\[\prod_{i=1}^{j+1}P_i\ll N\ll 2^{j+1}\prod_{i=1}^{j+1}P_i.\]
Since
we may now assume that $P_1\le N^{99/100}$ there will be an index $u$
such that 
\[N^{1/100}\le\prod_{i=1}^u P_i\le N^{99/100}.\]
Fixing such an index we split $n$ as $n=ab$ with
\[a=\prod_{i=1}^up_i,\;\;\;\mbox{and}\;\;\; b=\prod_{i=u+1}^{j+1}p_i, \]
so that $a\le N_1$ and $b\le N_2$ with
\[N_1:=2^{1+n_0}\prod_{i=1}^uP_i\;\;\;\mbox{and}\;\;\; 
N_2:=2^{1+n_0}\prod_{i=u+1}^{j+1}P_i,\] 
and hence
\begin{equation}\label{nbs}
N_1N_2\ll N^{1+\epsilon}\;\;\;\mbox{and}\;\;\; N_1,N_2\ll
N^{99/100+\epsilon}.
\end{equation}
We will then have $N_1 N^{-\epsilon}\ll a\le N_1$, and similarly for $b$.
Our description of $Q_j$ may now be expressed by requiring
that $a\in Q_{j,u}$ and $b\in Q'_{j,u}$ for appropriate sets $Q_{j,u}$
and $Q'_{j,u}$, together with the conditions that $ab\in I = (N', N'(1+\omega)] \cap [Y, x^{1/2 - \delta})$, that $p_{j+1}^{-1}ab<Y$, and that $p_{u+1}<p_u$.  Specifically, we take
$$Q_{j, u} = \{a = p_1...p_u: p_i \in (P_i, 2P_i], p_u < ...<p_1 <Y\},
$$and
$$Q_{j, u}' = \{b = p_{u+1}...p_{j+1}: p_i \in (P_i, 2P_i], p_{j+1} \in J, p_{j+1} < ...<p_{u+1} <Y\}.
$$


In order to separate the variables $a$ and $b$ completely we subdivide
the available ranges for $a,b,p_{j+1},p_u$ and $p_{u+1}$ into intervals of
the shape $(A,A+A/L]$, $(B,B+B/L]$, $(P_{j+1}',P_{j+1}'+P_{j+1}'/L]$, $(P_u',P_u'+P_u'/L]$
and $(P_{u+1}',P_{u+1}'+P_{u+1}'/L]$. Here the parameter $L$ will be a
small power of $N$.  The reader should note that we do not insist
that each of these intervals should have length 1 or more.  Indeed
such an interval may not contain any integers at all. There will be
$O(L^5(\log x)^2)$ collections of such intervals.  There will be some
for which the conditions $ab\in I$, $p_{j+1}^{-1}ab<Y$, and $p_{u+1}<p_u$
hold for every choice of $p_1,\ldots,p_{j+1}$ satisfying
\[a\in(A,A+A/L],\;\;\; b\in(B,B+B/L]\]
\[p_{j+1}\in(P_{j+1}',P_{j+1}'+P_{j+1}'/L],\;\;\; p_u\in (P_u',P_u'+P_u'/L],\;\;\;
p_{u+1}\in(P_{u+1}',P_{u+1}'+P_{u+1}'/L],\]
and
\[p_i\in I_i\;\;\;(i\not=1,u,u+1).\]
In this case the corresponding subsum is
\begin{equation}\label{eqn:largeprimebilinearsum}
\sum_{\substack{a\in Q_{j,u}\cap(A,A+N_1/L]\\
     p_u\in (P_u',P_u'+P_u'/L]}}
\;\;\sum_{\substack{b\in Q'_{j,u}\cap(B,B+B/L]\\
    p_{j+1}\in(P_{j+1}',P_{j+1}'+P_{j+1}'/L]\\ 
p_{u+1}\in(P_{u+1}',P_{u+1}'+P_{u+1}'/L]}}\lambda(ab),
\end{equation}
so that we have separated the variables $a$ and $b$.  For such sums we
can apply the following consequence of Friedlander
and Iwaniec \cite[Proposition 23.1]{FI}.
\begin{lem} \label{lem:JKbilinear}
Let $\mathfrak d_1(m)$ and $\mathfrak d_2(n)$ be bounded arithmetic
functions supported on $1\leq m \leq N_1$ and $1\leq n \leq N_2$
respectively.   Then
\begin{equation}
\sum_{m, n} \mathfrak d_1(m) \mathfrak d_2(n) \lambda(mn)
\ll (N_1+N_2)^{\frac{1}{12}} (N_1N_2)^{\frac{11}{12}+\epsilon}.
\end{equation}
\end{lem}

Since $A\le N_1$ and $B\le N_2$ this gives us a bound
\[\ll (N^{99/100})^{1/12}N^{11/12+\epsilon}=N^{1-1/1200+\epsilon},\]
in view of \eqref{nbs}.  Since there are $O(L^5N^{\epsilon})$ such
subsums the overall contribution is $O(L^5N^{1-1/1200+\epsilon})$.

It remains to consider the contribution from ``bad'' sets of ranges which are
not excusively contained in the region given by $ab\in I$,
$p_{j+1}^{-1}ab<Y$, and $p_{u+1}<p_u$. Suppose that the interval $I$ is
given by $I=[e_1,e_2]$, for example, and that there are integers
$a,a'\in(A,A+A/L]$ and $b,b'\in(B+B/L]$ for which $ab\in I$ but
$a'b'\not\in I$. Then we must have $ab=(1+O(L^{-1}))e_1$ or
$ab=(1+O(L^{-1}))e_2$. We now consider the total contribution from
integers $n\in Q_j$ for all such ``bad'' choices of intervals
$(A,A+A/L]$, $(B,B+B/L]$, $(P_{j+1}',P_{j+1}'+P_{j+1}'/L]$, $(P_u',P_u'+P_u'/L]$
and $(P_{u+1}',P_{u+1}'+P_{u+1}'/L]$.  Since each integer $n$ occurs
at most once, and $\lambda(n)=O(\tau(n))$, the contribution will be
\[\ll \sum_{n=(1+O(L^{-1}))e_i} \tau(n)\ll N^{1+\epsilon}L^{-1}.\]
Similarly, if we have $p_{j+1}^{-1}ab<Y$ but ${p_{j+1}'}^{-1}a'b'\ge Y$, then
$p_{j+1}^{-1}ab=(1+O(L^{-1}))Y$.  Now $P_{j+1}Y\asymp AB\le N_1N_2\ll
N^{1+\epsilon}$, so any $n$ which is to be counted will have a prime
factor $p\ll N^{1+\epsilon}/Y$ such that $p^{-1}n=(1+O(L^{-1}))Y$.
Thus, writing $n=pm$, we see that the total contribution in this case is
\[\ll \sum_{p\ll N^{1+\epsilon}/Y}\;\sum_{m=(1+O(L^{-1}))Y}\tau(pm)
\ll N^{1+\epsilon}Y^{-1}(1+L^{-1}Y)\ll N^{1+\epsilon}L^{-1},\]
for $L\le Y$.

Lastly, if $P_u=P_{u+1}$, then it may happen that the condition 
$p_{u+1}<p_u$ is satisfied by
some, but not all, pairs of primes $(p_u,p_{u+1})$ from the intervals
$(P_u',P_u'+P_u'/L]$ and $(P_{u+1}',P_{u+1}'+P_{u+1}'/L]$. Clearly
this problem cannot arise when $L\ge 2P_u$ since then the intervals
$(P_u',P_u'+P_u'/L]$ and $(P_{u+1}',P_{u+1}'+P_{u+1}'/L]$ contain at
most one prime each. It follows that any $n$ to be counted in this
case will have two prime factors $p'>p\ge P_u\ge L/2$ with
$p'=(1+O(L^{-1}))p$. Hence the corresponding contribution is
\[\ll\sum_{\substack{p'>p\ge L/2\\  p'=(1+O(L^{-1}))p}}\;
\sum_{\substack{n\ll N\\ p'p\mid n}}\tau(n)
\ll\sum_{\substack{p'>p\ge L/2\\  p'=(1+O(L^{-1}))p}}
\frac{N^{1+\epsilon}}{p'p}\ll N^{1+\epsilon}L^{-1}.\]

We therefore find that our sum is
\[\ll L^5N^{1-1/1200+\epsilon}+N^{1+\epsilon}L^{-1}\]
if $L\le Y$.
We may then choose $L=N^{1/10000}$ for example, to achieve the claimed
power saving.

\subsection{Small $d$}
To handle small $d$ it will be enough to show for any
$f\leq F$, $d\leq D_1$, and any non-principal $\chi\modd{d}$, that 
\begin{equation}
\sumstar_{c\modd{df}}\sumb_{\substack{(z_1,z_2)\in\U_1\times\U_2\\ cz_2\equiv
      z_1\modd{df}}} 
\beta_{z_1}\beta_{z_2}\chi(c) \ll_C \frac{N^2}{\log^C x},
\end{equation}
for every $C>0$. Since 
\[\sumstar_{c\modd{df}} \chi(c) = 0,\]
it suffices to prove that if $\U=\U_1$ or $\U_2$ then
there is an $\mathfrak M = \mathfrak M(\U, df)$ such that
\begin{equation}\label{eqn:smalldmain}
\sum_{\substack{z\in\U\\ z \equiv \alpha \modd{df}}} 
\beta_z = \mathfrak M + O\bfrac{N}{\log^C x},
\end{equation}
for any $(\alpha,df)=1$, and any $C>0$.

As in Section \ref{sec:larged}, we may assume that
$\beta_z=\beta_{N(z)}$, where $\beta_n$ is the
indicator function of either $Q_j$ or $R$. We describe the procedure
for $Q_j$, the method for $R$ being similar.  We decompose $z$ as $z_1z_2$
with $N(z_1)$ being the largest prime factor of $N(z_1z_2)$. The
requirement that $n\in Q_j$ is then equivalent to a condition of the
form $N(z_2)\in
Q_j'$ together with a restriction of the type $N(z_1)\in I(z_2)$ for
some real interval $I(z_2)$. Specifically we have
\[Q_j'=\{p_2...p_{j+1}: p_{j+1} \in J,\; p_{j+1} < ... < p_2\}\]
and
\[I(z_2)=(p_2,Y)\cap\left(\frac{N'}{N(z_2)}\,,\,
\frac{N'(1+\omega)}{N(z_2)}\right]
\cap\left[\frac{Y}{N(z_2)}\,,\,\frac{x^{1/2-\delta}}{N(z_2)}\right),\]
where $p_2$ is the largest prime factor of $N(z_2)$.  When $\U$ is given
by \eqref{Uf} the condition on the size of $N(z_1z_2)$ is exactly the
condition
\[N(z_1)\in\left(\frac{N'}{N(z_2)}\,,\,\frac{N'(1+\omega)}{N(z_2)}\right],\]
and we have $\theta_0<\arg z<\theta_0+\omega_2$ exactly when
\[\theta_1(z_2)<\arg z_1<\theta_1(z_2)+\omega_2,\]
with $\theta_1(z_2)=\theta_1-\arg z_2$.

It follows that
\begin{equation}\label{sN}
\sum_{\substack{z\in\U\\ z \equiv \alpha \modd{df}}} \beta_z
=\sum_{\substack{z_2\in Q_j'\\ (z_2,df)=1}}\N(z_2,\alpha),
\end{equation}
where $\N(z_2,\alpha)$ is the number of Gaussian integers $z_1$ satisfying
\[z_1 \equiv \alpha\overline{z_2} \modd{df},\;\;\; N(z_1)\in  I(z_2),
\;\;\;\mbox{and} \;\;\; \theta_1(z_2)<\arg z_1<\theta_1(z_2)+\omega_2,\]
and for which $N(z_1)$ is prime.  Here $\overline{z_2}$ is the inverse
of $z_2$ modulo $df$.

We can estimate $\N(z_2,\alpha)$ using a
form of the Prime Number Theorem for arithmetic progressions, over
number fields.
Given $q\in\mathbb{N}$, any Gaussian integer $\alpha$ coprime to $q$,
and any $\theta\in[0,2\pi]$ write $\pi(x;q,\alpha,\theta)$ for the
number of Gaussian primes $\mu\equiv\alpha\modd{q}$ of norm at most $x$
and with $0\le\arg(\mu)\le \theta$.  The principal result of Mitsui
\cite{mitsui} tells us that there is an absolute constant $c$ such
that
\begin{equation}\label{eqn:Mitsui}
\pi(x;q,\alpha,\theta)=\frac{4}{\phi_{\mathbb{Q}(i)}(q)}
\frac{\theta}{2\pi}{\rm Li}(x)+O_A(x\exp(-c\sqrt{\log x}))
\end{equation}
uniformly for all $\theta\in[0,2\pi]$ and all $q\le(\log x)^A$.  Here
$\phi_{\mathbb{Q}(i)}(q)$ is the Euler $\phi$-function for the
Gaussian integers.

We now apply \eqref{eqn:Mitsui} to estimate $\N(z_2,\alpha)$.  We have
$I(z_2)\subseteq(0,2N/N(z_2)]$, and so we will need to know that
$df\le (\log 2N/N(z_2))^A$ for some constant $A$.  However we
recall that if $p$ divides an element of $Q_j$ then
one has $p\ge x^{\delta}$ with $\delta=(\log x)^{\varpi-1}$.  Thus we
will have $2N/N(z_2)\ge x^{\delta}$ so that $df\le (\log
2N/N(z_2))^{C/\varpi}$ whenever $df\le(\log x)^C$.  The required
condition is therefore satisfied when $f\le F$ and $d\le D_1$.

We therefore find that
\[\N(z_2,\alpha)=\mathfrak{M}(z_2,df,j,\U)+
O\left(\frac{N}{N(z_2)}\exp(-c(\log x)^{\varpi/2})\right),\]
where the main term $\mathfrak{M}(z_2,df,j,\U)$ is, crucially,
independent of $\alpha$. If we feed this into \eqref{sN} we then
obtain the desired estimate \eqref{eqn:smalldmain}.  This completes
our treatment of small $d$.

\section{Distribution of sequences in arithmetic 
progressions}\label{section:BDH}
The purpose of this final section is to prove Corollary \ref{cor:SWsequence}
and the more general Theorem \ref{thm:2} below.  Theorem 2 is
motivated by the possibility that an interesting arithmetic function
can be biased for certain small moduli --- that is, the sequences does
not satisfy the Siegel--Walfisz condition.  In this case, we can still
prove a result about the distribution of such a sequence in arithmetic
progressions by adjusting the main term.  

We first fix some notation for the rest of the section: for any
arithmetic function $a(n)$ with finite support, we let  
\begin{equation}\label{eqn:ataudef}
\|a \tau\|^2 = \sum_{n} |a(n)|^2 \tau(n)^2.
\end{equation}
For arithmetic functions $c_1$ and $c_2$, we shall see that $\|c_1
\tau\|^2\|c_2 \tau\|^2$ is an upper bound for $\|c_1*c_2\|^2$.   

Let $\ga(n)$ and $\de(n)$ be arithmetic functions supported on $n\leq
N_1$ and $n\leq N_2$ respectively and let $Q_0\geq 1$.  For a
character $\chi$, let $Q(\chi)$ be the conductor of the unique
primitive character which induces $\chi$.  Now, 
let 
\begin{equation}
S(a, q) = \sum_{\substack{n_1\equiv an_2 \modd{q}
\\(n_1n_2, q)=1}} \ga(n_1)\de(n_2), 
\end{equation}
\begin{equation}
\M(a, q) = \frac{1}{\phi(q)} \sum_{n_1,
  n_2}\ga(n_1)\de(n_2)\sum_{\substack{\chi \modd{q} \\ Q(\chi) \leq
    Q_0}} \chi(n_1) \overline{\chi(an_2)} 
\end{equation}and
\begin{equation}
\E(a, q) = S(a, q) - \M(a, q).
\end{equation}Note that $\M(a, q)$ is the expected main term for
$S(a, q)$.  The main result of this section is below. 

\begin{thm}\label{thm:2}
For any $Q \in \mathbb{N}$, let
\begin{equation}
\E = \sum_{q \leq Q}\; \sumstar_{a \modd{q}} |\E(a, q)|^2.
\end{equation}
Then,
\begin{equation}
\E \ll \left(Q+\frac{N_1N_2}{Q_0}\right)
(\log Q) \|\ga\tau \|^2 \|\de\tau \|^2. 
\end{equation}
\end{thm}

\subsection{Proof of Theorem \ref{thm:2}}
We first prove a consequence of the large sieve which lies at the heart
of our result. 
\begin{prop}\label{proplargesieve}
Let $a(n)$ be any arithmetic function supported on $n\leq N$, and for
any character $\chi$, let 
\begin{equation}
A(\chi) = \sum_n a(n) \chi(n).
\end{equation}Further, let $Q(\chi)$ denote the conductor of the
unique primitive character that induces $\chi$.  Then 
\begin{equation}
S := \sum_{h\leq H} \frac{1}{\phi(h)} \sum_{\substack{\chi \modd{h}
\\ Q(\chi) > h_0}} |A(\chi)|^2  
\ll \left(H + \frac{N}{h_0}\right) (\log H)\|a\|^2.
\end{equation}
\end{prop}
\begin{proof}
For each non-principal character $\chi \modd{h}$,
let $\psi \modd{h_1}$ be the unique primitive character which induces
$\chi$, where we may write $h = h_1h_2$ for $h_1>1$.  Then for any
$n$, $\chi(n) = \psi(n)$ if $(n, h_2) = 1$ and $\chi(n) = 0$
otherwise.  Let us write  
\begin{equation}
A(\psi, h_2) = \sum_{\substack{n\\ (n, h_2) = 1}} a(n)\psi(n),
\end{equation}
whence
\begin{align} \label{eqn:S1}
S \leq \sum_{\substack{h_1 h_2 \leq H\\ h_1\geq h_0}}
 \frac{1}{\phi(h_2)} \frac{1}{\phi(h_1)}\sumstar_{\psi \modd{h_1}} |A(\psi, h_2)|^2,
\end{align}where $\sumstar_{\psi \modd{h_1}}$ denotes a sum over all
primitive characters modulo $h_1$.  Applying the multiplicative large
sieve (see e.g. (9.52) in \cite{FIO}) in dyadic ranges, we see that 
\begin{equation}
 \sum_{h_0\le h_1 \le H/h_2}
\frac{1}{\phi(h_1)} \sumstar_{\psi \modd{h_1}} |A(\psi, h_2)|^2 \ll
\left(\frac{H}{h_2} + \frac{N}{h_0}\right) \|a\|^2, 
\end{equation}whence
\begin{align*}
S &\ll \sum_{h_2\leq H} \frac{1}{\phi(h_2)} \left(\frac{H}{h_2} +
  \frac{N}{h_0}\right) \|a\|^2 \notag \\ 
&\ll \left(H + \frac{N\log H}{h_0}\right) \|a\|^2. 
\end{align*}
The Proposition then follows.
\end{proof}

We now proceed to the proof of Theorem \ref{thm:2}.
\begin{proof}
We have
\begin{align*}
\E(a, q)
&= \frac{1}{\phi(q)} \sum_{\substack{\chi \modd{q} \\
    Q(\chi)>Q_0}} G(\chi)\overline{D(\chi)} \overline{\chi(a)}, 
\end{align*}where
\begin{equation}
G(\chi) = \sum_{n} \ga(n)\chi(n),
\end{equation}and
\begin{equation}
D(\chi) = \sum_{n} \overline{\de(n)}\chi(n).
\end{equation}
Then
\begin{align}
\E &= \sum_{q\leq Q}\frac{1}{\phi(q)^2}
\sumstar_{a\modd{q}} \left|\sum_{\substack{\chi \modd{q} \\ Q(\chi)>Q_0}}
  G(\chi)\overline{D(\chi) \chi(a)}\right|^2 \notag \\ 
  &= \sum_{q\leq Q}\frac{1}{\phi(q)}
\sum_{\substack{\chi \modd{q} \\ Q(\chi)>Q_0}}
    \left|G(\chi)\overline{D(\chi) }\right|^2 \notag \\ 
&\ll \left(Q + \frac{N_1N_2}{Q_0}\right)(\log Q)\|a\|^2 ,
\end{align}
where
\begin{equation}
a(n) = \sum_{n = n_1n_2} \ga(n_1)\overline{\de(n_2)},
\end{equation}
and where we have used Proposition \ref{proplargesieve}
with $a(n)$ supported on $n\leq N_1N_2$. 

However Cauchy's inequality yields
\[|a(n)|^2\le \tau(n)\sum_{n = n_1n_2} |\ga(n_1)\de(n_2)|^2
\le \sum_{n = n_1n_2} |\ga(n_1)\de(n_2)|^2\tau(n_1)\tau(n_2),\]
whence
\[\|a\|^2\le \sum_{n_1,n_2}|\ga(n_1)\de(n_2)|^2\tau(n_1)\tau(n_2)
\le \|\ga\tau \|^2 \|\de\tau \|^2.\]
It follows that
\begin{align}
\E
&\ll \left(Q+\frac{N_1N_2}{Q_0}\right)(\log Q) \|\ga\tau \|^2 \|\de\tau \|^2, 
\end{align}as desired.

\end{proof}

\subsection{Proof of Corollary \ref{cor:SWsequence}}

We have 
\begin{equation}
S(x; a, q) = \frac{1}{\phi(q)} \sum_{\chi \modd{q}}
\sum_{n_1, n_2} \chi(n_1)\overline{\chi(an_2)}c_1(n_1)c_2(n_2). 
\end{equation}

Then, taking $Q_0 = (\log x)^{A + 1}$, we see that
\begin{eqnarray*}
\lefteqn{ |S(x; a, q) - S(x; q)|^2}\\
&=&  \left| \frac{1}{\phi(q)} \sum_{\substack{\chi \modd{q} 
\\ \chi\not=\chi_0}}\sum_{n_1, n_2}
  \chi(n_1)\overline{\chi(an_2)} c_1(n_1)c_2(n_2)\right|^2\\ 
&= & \left|\E(a, q) + E(a, q, Q_0)\right|^2\\
&\ll& |\E(a, q)|^2 + |E(a, q, Q_0)|^2,
\end{eqnarray*}where
\begin{equation}
E(a, q, Q_0) = \frac{1}{\phi(q)}\sum_{\substack{\chi \modd{q} \\ 1<Q(\chi)
    \leq Q_0}}\left\{ \sum_n c_1(n) \chi(n)\right\}\left\{\sum_n
c_2(n)\overline{\chi(an)}\right\} 
\end{equation}
and $\E(a, q)$ is as defined in Theorem \ref{thm:2} with
$\gamma = c_1$ and $\delta = c_2$.  We let $B(A) = A+1$ and apply
Theorem \ref{thm:2} to see that 
\begin{equation}
\sum_{q\leq Q} \sumstar_{a\modd{q}} |S(x; a, q) - S(x; q)|^2
\ll \frac{x^2}{\log^A x} \|c_1\tau\|^2\|c_2\tau\|^2 + \sum_{q\leq Q}
\sumstar_{a\modd{q}} |E(a, q, Q_0)|^2.
\end{equation}
It now suffices to bound the last term.  
We apply the Siegel--Walfisz condition \eqref{eqn:SW1} with $Q_0
\le (\log x)^\kappa$, for some constant $\kappa \geq A+1$ to be determined,
and deduce  that
\begin{align}\label{eqn:Cchid}
\sum_{n} c_1(n) \chi(n) \ll x^{1/2} \|c_1 \|(\log x)^{-\kappa}.
\end{align}
The sum involving $c_2$ may be bounded trivially as
\[\sum_n |c_2(n)|\le x^{1/2}\| c_2\|,\]
and since the number
of characters with modulus less than $Q_0$ is at most
$Q_0^2$, we have 
\begin{equation}
E(a, q, Q_0) \ll \frac{Q_0^2 x \|c_1 \| \|c_2\|}{\phi(q)(\log x)^{\kappa}}.
\end{equation}
Thus
\begin{align*}
\sum_{q \le Q}\sumstar_{a\modd{q}} |E(a, q, Q_0)|^2 
&\ll \frac{Q_0^4 x^2\|c_1 \|^2\|c_2\|^2}{(\log x)^{2\kappa}} \sum_{q\le Q}
\frac{1}{\phi(q)}\\
&\ll \frac{Q_0^4 x^2 \|c_1 \|^2\|c_2\|^2}{(\log x)^{2\kappa-1}}.
\end{align*}
We then see that the Corollary
follows upon choosing 
\[2\kappa = 5A+5.\]

\end{document}